\documentclass[11pt]{amsart}
\usepackage{amsmath,amsthm,amssymb,latexsym,color}
\usepackage{hyperref}
\usepackage{lipsum}
\usepackage{todonotes}
\usepackage{multirow}
\usepackage[mathscr]{eucal}

\usepackage{xcolor}
\date{}
\usepackage{pgfplots}
\usepackage{tikz}
\usetikzlibrary{arrows,matrix,positioning}
\usepackage{graphicx}
\usepackage{subfigure}

\title{A remark on the Ramsey number of the hypercube}
\author{
Konstantin Tikhomirov
}
\address{
\medskip
\noindent
School of Mathematics\\
Georgia Institute of Technology\\
686 Cherry street\\
Atlanta, GA 30332, and}
\address{
\medskip
\noindent
Department of Mathematical Sciences\\
Carnegie Mellon University\\
Wean Hall 6113\\
Pittsburgh, PA 15213\\
\texttt{\small
e-mail:   ktikhomi@andrew.cmu.edu}
}

\thanks{The work is partially supported by the NSF Grant DMS 2054666}

\newtheorem{theorem}{Theorem}[section]
\newtheorem*{theorem*}{Theorem}

\newtheorem{lemma}[theorem]{Lemma}
\newtheorem{cor}[theorem]{Corollary}
\newtheorem{defi}[theorem]{Definition}
\newtheorem{prop}[theorem]{Proposition}

\theoremstyle{definition}
\newtheorem{Remark}[theorem]{Remark}
\newtheorem*{Remark*}{Remark}

\theoremstyle{plain}

\newcommand{\N}{\mathbb{N}}

\newcommand{\Exp}{\mathbb{E}}

\newcommand{\Event}{\mathcal{E}}
\def\Prob{{\mathbb P}}

\def\neigh{{\mathcal N}}
\def\cneigh{{\mathcal C\mathcal N}}

\begin{document}

\maketitle

\begin{abstract}
A well known conjecture of Burr and Erd\H os asserts that the Ramsey number $r(Q_n)$ of the 
hypercube $Q_n$ on $2^n$ vertices is of the order $O(2^n)$.
In this paper, we show that $r(Q_n)=O(2^{2n-c n})$
for a universal constant $c>0$, improving upon the previous best known bound 
$r(Q_n)=O(2^{2n})$, due to Conlon, Fox and Sudakov.
\end{abstract}

\section{Introduction}

The {\it Ramsey number} $r(H)$ of a graph $H$ is the smallest integer $N\in\N$ such that
any two-coloring of the complete graph on $N$ vertices contains a monochromatic copy of $H$.
Given an integer $n$, denote by $Q_n$ the $n$--dimensional hypercube viewed as a graph
(where the edge set is formed by the ``geometric'' edges of the hypercube).
Burr and Erd\H os \cite{BE1975} conjectured that the Ramsey number of $Q_n$
is of order $O(2^n)$ (where the implicit constant is absolute).
Improvements of the trivial bound $r(Q_n)\leq r(K_{2^n})$ were obtained by Beck \cite{B1983}, Graham, R\"{o}dl and Ruci\'{n}ski \cite{GRR2001}, Shi \cite{S2001,S2007}, Fox and Sudakov \cite{FS2009},
Conlon, Fox and Sudakov \cite{CFS2016}, as well as Lee \cite{Choongbum}.
Disregarding multiplicative constants, the best upper bound $r(Q_n)=O(2^{2n})$ prior to our work was obtained in \cite{CFS2016} using the dependent random choice, and it applies to arbitrary bipartite graphs with a given maximum degree.
\begin{theorem*}[{\cite[Theorem~4.1]{CFS2016}}]
For every bipartite graph $H$ on $m$ vertices with maximum degree $d$, one has $r(H)\leq 2^{d+6}m$.
\end{theorem*}
Let us remark at this point that, as was shown in \cite{GRR2000}, for every $d\geq 2$ and $m\geq d+1$
there exists a bipartite graph $H$
on $m$ vertices with the maximum degree at most $d$ and such that $r(H)\geq 2^{c'd}m$ for some constant $c'>0$.
Therefore, a proof of the aforementioned conjecture of Burr and Erd\H os or even a weaker bound $r(Q_n)=2^{n+o(n)}$
should necessarily make use of properties of the hypercube other than the size of its vertex set
and the vertex degrees.

The above theorem follows immediately as a corollary of an embedding theorem in the same work \cite{CFS2016}
(we provide a slightly simplified version here).
\begin{theorem*}[{\cite[Theorem~4.7]{CFS2016}}]
Let $H$
be a bipartite graph on $m$ vertices with maximum degree
$d\geq 2$.
If $G$ is a bipartite graph with edge density $\alpha\in(0,1]$ and at least $16d^{1/d}\alpha^{-d} m$
vertices in each part, then $H$ is a subgraph of $G$.
\end{theorem*}
The {\it dependent random choice} is a well established technique in extremal combinatorics
(see, among others, papers
\cite{AKS2003,CFS2016,FS2009,G1998,KR2001,S2003} as well as survey \cite{FS2011})
which allows us to generate collections of graph vertices with many common neighbors.
In the context of bipartite graph embeddings, the use of the dependent random choice
can be roughly outlined as follows (the technical details which we omit here will be considered later in the paper
with proper rigor).
\begin{itemize}
\item[\bf(I)] Given an ambient bipartite graph $G=(V^{up}_{G},V^{down}_{G},E_{G})$ with the vertex set
$V^{up}_{G}\sqcup V^{down}_{G}$, one constructs a collection $S$ ($|S|\geq m$) of vertices in $V^{up}_{G}$
such that for an $1-o(1/m)$ fraction of
$d$--tuples of vertices $(v_1,\dots,v_d)$ from $S$, the number of common neighbors of $(v_1,\dots,v_d)$
in $V^{down}_{G}$ is at least $m$. The construction procedure for the set $S$ with such properties
is a variation of the first moment method.
\item[\bf(II)] For a bipartite graph $H=(V^{up}_{H},V^{down}_{H},E_{H})$ on $m$ vertices, 
$V^{up}_{H}$ is mapped into a random $|V^{up}_{H}|$--tuple
of vertices of $G$ uniformly distributed on the set of $|V^{up}_{H}|$--tuples of distinct elements of $S$.
The assumption on $S$ from {\bf (I)} then guarantees that with probability close to one it is possible to find an injective mapping
$V^{down}_{H}\to V^{down}_{G}$ which preserves vertex adjacency thus producing an embedding of $H$ into $G$.
The injective mapping can be constructed iteratively, by embedding one vertex of $V^{down}_{H}$ at a time.
\end{itemize}
Whereas the above scheme is sufficient to prove that $r(Q_n)\leq 2^{2n+o(n)}$ (see Corollary~\ref{aljfhbefwhbfoufbowfub} and
Remark~\ref{aljehbfoweufb}), the stronger result
of \cite[Theorem~4.1]{CFS2016} requires some additional ingredients which are not discussed here.

\medskip

Note that if we do not
impose any {\it structural} assumptions on the common neighborhoods $\cneigh_{G}(v_1,\dots,v_d)$, $v_1,\dots,v_d\in S$, 
then the condition that $|\cneigh_{G}(v_1,\dots,v_d)|=\Omega(|V^{down}_{H}|)$ for a majority
of $d$--tuples in $S$ becomes essential for the step {\bf (II)} to work through.
More specifically, if the graph $G$ and sets $S\subset V^{up}_{G}$
and $T\subset V^{down}_{G}$ are such that
\begin{equation}\label{aoieuf49873ylakslbf}
\cneigh_{G}(v_1,\dots,v_d)\subset T\quad
\mbox{for most choices of $d$--tuples of distinct vertices $v_1,\dots,v_d\in S$}
\end{equation}
then
for any $d$--regular graph $H=(V^{up}_{H},V^{down}_{H},E_{H})$,
the random mapping $V^{up}_{H}\to S$ in {\bf (II)} can be extended
to an embedding of $H$ into $G$ with high probability
{\it only if } $|T|=\Omega(|V^{down}_{H}|)$.
The next example shows that conditions of type \eqref{aoieuf49873ylakslbf} hold for certain bipartite graphs, 
and thus the dependent random choice technique for bounding $r(Q_n)$ 
as outlined in {\bf (I)--(II)} hits a barrier around $2^{2n}$. 
We only sketch the construction here; its fully rigorous description is provided
for an interested reader in Appendix~\ref{lfuhb93476ajshbflj}.



Let $\varepsilon>0$ be an arbitrary small constant,
and consider a {\it random} bipartite graph $\Gamma=(V^{up}_{\Gamma},V^{down}_{\Gamma},E_{\Gamma})$
with $|V^{up}_{\Gamma}|=|V^{down}_{\Gamma}|= 2^{2n-\varepsilon n}$ in which the set $V^{down}_{\Gamma}$
is partitioned into $2^{n-\varepsilon n/2}$ subsets $V^{down}_{\Gamma}(i)$, $1\leq i\leq 2^{n-\varepsilon n/2}$,
of size $2^{n-\varepsilon n/2}$ each. Assume that
each vertex $v$ in $V^{up}_{\Gamma}$ is adjacent to all the vertices in
$$
\bigcup_{i\in I_v}V^{down}_{\Gamma}(i),
$$
where $I_v$ is a uniform random $2^{n-\varepsilon n/2-1}$--subset of $\{1,\dots,2^{n-\varepsilon n/2}\}$
and where $I_v$, $v\in V^{up}_{\Gamma}$, are mutually independent. Thus, $\Gamma$ has edge density $1/2$.
It can be shown that, assuming $n$ is large, with probability close to one $\Gamma$ has the following property:
for {\it every} subset $S$ of $V^{up}_{\Gamma}$ with $|S|\geq 2^{n-1}$ there is a subset $T\subset V^{down}_{\Gamma}$
with $|T|\leq 2^{n-\varepsilon n/4}$ such that for at least half of $n$--tuples
of vertices in $S$, the common neighborhood of the $n$--tuple is contained in $T$.
Conditioned on such realization of $\Gamma$, the randomized embedding of $Q_n$
into $\Gamma$ as described in {\bf (I)--(II)} will fail with probability close to one, regardless of how the set $S$ is constructed.
The deficiency of the scheme {\bf (I)--(II)} is thus not in choosing the set $S$ but in embedding $V^{up}_{H}$
into $S$ {\it uniformly} at random.

\medskip

In view of the last remarks, the questions of interest are which structural properties of the hypercube
should play a role in bounding its Ramsey number and whether the randomized embedding
scheme {\bf (I)--(II)} could be modified accordingly to accommodate those properties.
In this paper, we make some progress on those questions. The main result is
\begin{theorem}\label{thmain}
There are universal constants $n_0,c>0$ such that for every $n\geq n_0$,
and every bipartite graph $G=(V^{up}_{G},V^{down}_{G},E_{G})$ satisfying $|V^{up}_{G}|,|V^{down}_{G}| \geq 2^{2n-cn}$
and $|E_{G}|\geq \frac{1}{2}|V^{up}_{G}|\,|V^{down}_{G}|$,
the hypercube $Q_n$ can be embedded into $G$.
\end{theorem}
\begin{Remark*}
Our proof shows that one can take $c=0.03656$ assuming that $n_0$ is sufficiently large.
\end{Remark*}
As an immediate corollary, we get
\begin{cor}\label{alijfr0387yajklb}
Let $n_0,c>0$ be as in the last theorem. Then for every $n\geq n_0$,
the Ramsey number of the hypercube $Q_n$ satisfies $r(Q_n)\leq  2^{2n-cn+1}+2$.
\end{cor}
\begin{Remark}\label{aljehbfoweufb}
Corollary~\ref{alijfr0387yajklb} can be obtained from Theorem~\ref{thmain} via the following standard construction.
Let $N$ be the largest even integer not exceeding $2^{2n-cn+1}+2$, fix any $2$--coloring of edges of $K_N$,
and take any equipartition $V^{up},V^{down}$ of its vertex set, so that $|V^{up}|,|V^{down}| \geq 2^{2n-cn}$.
Then the monochromatic bipartite subgraph of $K_N$ corresponding to the majority color among
the edges connecting $V^{up},V^{down}$, satisfies the assumptions of Theorem~\ref{thmain}.
\end{Remark}

\medskip

We build the discussion of the proof ideas on the example of the random ambient graph $\Gamma$
considered above.
To avoid technical details whenever possible, we postulate that $\varepsilon n/2$ is an integer,
where $\varepsilon>0$ is a small constant.
Although the procedure {\bf (I)--(II)} cannot produce an embedding of $Q_n$ into $\Gamma$
with high probability, the embedding strategy can be modified in this setting as follows.
Take $w:=\varepsilon n/2$.
Let 
$$\mathcal T:=\big\{v\in \{-1,1\}^n:\;\mbox{vector $v$ has an odd number of $-1$'s}\big\},$$
and let $(\mathcal T_b)_{b\in\{-1,1\}^w}$ be a partition of $\mathcal T$, where
$$
\mathcal T_b:=\big\{v\in \mathcal T:\,(v_{n-w+1},\dots,v_n)=b\big\},\quad b\in \{-1,1\}^w.
$$
Observe that for every admissible $b$, the set $\mathcal T_b$ has cardinality $2^{n-w-1}$.
Now, we construct a random injective mapping $f$ of $\mathcal T$ into $V^{down}_{\Gamma}$.
Let $\phi:\{-1,1\}^w\to[2^w]$ be an arbitrary fixed bijection.
Then for every $b\in \{-1,1\}^w$ we let $f((v:\,v\in \mathcal T_b))$ be a random $2^{n-w-1}$--tuple
uniformly distributed on the collection of all $2^{n-w-1}$--tuples of distinct
elements of $V^{down}_{\Gamma}(\phi(b))$, and we require that
the random vectors $f((v:\,v\in \mathcal T_b))$, $b\in\{-1,1\}^w$
are mutually independent.
Thus, we split the part $\mathcal T$ of the hypercube vertices into blocks
according to which $(n-w)$--facet $\{v\in \{-1,1\}^n:\,(v_{n-w+1},\dots,v_n)=b\}$
a vertex belongs to, and then embed those blocks into corresponding sets $V^{down}_{\Gamma}(\phi(b))$.
Note that for every vertex $v^{up}$ in the part $\{-1,1\}^n\setminus \mathcal T$ of the hypercube, the majority (specifically, $n-w$)
of its neighbors belong to a single
block, and only the small number $w$ of its neighbors reside in other blocks.
Since all vertices within a given set $V^{down}_{\Gamma}(\phi(b))$ have the same set
of neighbors in $V^{up}_{\Gamma}$, one may expect that with high probability the set of common neighbors
for $f(\neigh_{Q_n}(v^{up}))$ in $\Gamma$ is of order $\Omega(2^{2n-\varepsilon n-1}\cdot 2^{-w-1})\gg 2^n$ (where the multiple $2^{-w-1}$ appears since we intersect $w+1$ random $2^{2n-\varepsilon n-1}$--subsets
of $V^{up}_{\Gamma}$).
Provided that the probability of that event is sufficiently close to one, by taking the union bound this would imply that with high probability the mapping $f:\mathcal T\to V^{down}_{\Gamma}$
can be extended to an embedding of $Q_n$ into $\Gamma$. We emphasize that the above description does not serve as a mathematical proof, but only highlights the idea of how the
randomized ``block'' embedding will be constructed. See Figure~\ref{ajnfeofhjnojfnoijfno} for an illustration.

\begin{figure}[h]\label{ajnfeofhjnojfnoijfno}
\caption{An illustration of the ``block'' embedding of the cube into the bipartite graph $\Gamma$. The lower part of the vertex set of $\Gamma$ is partitioned into blocks, so that within each block
every vertex has a same set of neighbors. The cube is embedded into $\Gamma$ so that each collection of vertices $\mathcal T_b$, $b=\pm 1$ is mapped into a single block. In the picture, the vertices of the cube which belong to $\mathcal T$ are enlarged,
and the dotted ellipses mark the facets of the cube corresponding to the blocks of vertices. The edge density of $\Gamma$ in this illustration is made greater than $1/2$ in view of the small number of blocks (otherwise,
no block embedding would be possible).}

\centering  

\includegraphics[width=0.8\textwidth]{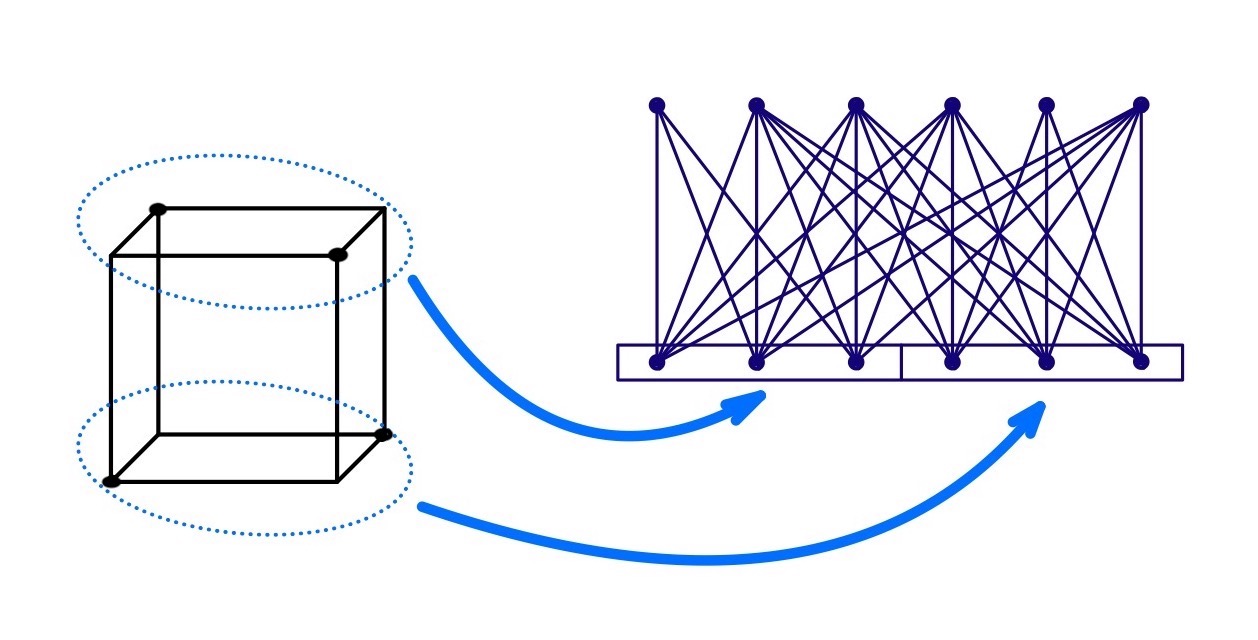}

\end{figure}

The proof of Theorem~\ref{thmain} utilizes ``block'' embeddings
similar to the above (although more technically involved), and is based on the following alternatives.
Given a bipartite graph $G$ on $2^{2n-cn}+2^{2n-cn}$ vertices of density $1/2$, at least
one of the following three conditions holds:
\begin{itemize}
\item[(a)] The dependent random choice method from {\bf(I)--(II)}
succeeds in embedding the hypercube $Q_n$ into the ambient graph $G$ since the common neighborhoods
of $n$--tuples of vertices in the set $S$ from {\bf(I)} tend to have small overlaps.
\item[(b)] $G$ contains a large subgraph of a density significantly larger than $1/2$.
In this case, we apply the dependent random choice to that subgraph of $G$ instead of $G$ itself to construct
the hypercube embedding, again according to the scheme {\bf(I)--(II)}.
\item[(c)] $G$ contains a subgraph having a ``block'' structure similar to that of
the random graph $\Gamma$ from the above example. In this case,
we construct a special randomized facet-wise embedding of the hypercube.
\end{itemize}
The structural part of this trichotomy (not implementing an actual embedding
of $Q_n$ and only concerned with the properties of the ambient graph $G$) is formally stated as
Proposition~\ref{vviyagcvuytcvuqtycv}.
This proposition is then combined with appropriate embedding procedures to obtain the main result of the note.

Below is the outline of the paper.

In Section~\ref{aigfviygvciygv}, we revise the notation used in this paper, and
recall a standard concentration inequality for independent Bernoulli random variables.

In Section~\ref{s:dep}, we provide a rigorous
description of the dependent random choice technique as outlined in {\bf(I)--(II)}.
The material for this section is taken, with some minor alterations, from \cite{FS2009,FS2011},
and is applied to deal with the cases (a) and (b) above.

In Section~\ref{s:disp}, we introduce the notion of {\it $(p,M)$--condensed} common neighborhoods
of tuples of graph vertices
and show that existence of a non-condensed collection of neighborhoods (with appropriately chosen parameters)
guarantees that the dependent random choice method produces an embedding of $Q_n$ into $G$.
The main result of this section --- Lemma~\ref{fljknfefjwfpiwjff} --- is used to treat the case (a).

In Section~\ref{s:blockstructured} we define {\it block-structured} bipartite graphs
and develop a randomized procedure for embedding the hypercube into graphs of that type.
That embedding procedure is crucial for the case (c).

Finally, in Section~\ref{s:dich} we state and prove the structural Proposition~\ref{vviyagcvuytcvuqtycv},
and complete the proof of the main result of the paper.

\bigskip

{\bf Acknowledgements.} The author would like to thank Han Huang for valuable discussions.

\section{Notation and preliminaries}\label{aigfviygvciygv}

The notation $[m]$ will be used for a set of integers $\{1,2,\dots,m\}$.
We will denote constants by $c,C$ etc. Sometimes we will add a subscript
to a name of a constant to assign it to an appropriate statement within the paper.
For example, the universal constant $c_{\text{\tiny\ref{kjfaoifboqfuifhbluhb}}}\in(0,1]$
is taken from Lemma~\ref{kjfaoifboqfuifhbluhb}.

We will occasionally use standard asymptotic notations $o(\cdot)$ and $\Omega(\cdot)$, for example, for two functions $f$ and $g$ on positive integers,
$f(n)=o(g(n))$ if $\lim\limits_{n\to\infty}\frac{f(n)}{g(n)}=0$ and $f(n)=\Omega(g(n))$ if $\liminf\limits_{n\to\infty}\frac{|f(n)|}{|g(n)|}>0$.

In this note, any bipartite graph $G$ is viewed as a triple $(V^{up}_{G},V^{down}_{G},E_{G})$,
where $V^{up}_{G}$ and $V^{down}_{G}$ are sets of ``upper'' and ``lower'' vertices, and $E_G$
is a collection of edges connecting vertices in $V^{up}_{G}$ to those in $V^{down}_{G}$.

Given a bipartite graph $G=(V^{up}_{G},V^{down}_{G},E_{G})$, its {\it edge density}
is defined as the ratio
$$
\frac{|E_{G}|}{|V^{up}_{G}|\,|V^{down}_{G}|}.
$$

For a vertex $v$ in a graph $G$, the set of neighbors of $v$ in $G$ will be denoted by $\neigh_G(v)$.
Further, given a collection of vertices $\{v^{(1)},\dots,v^{(r)}\}$, the set of their
common neighbors will be denoted by $\cneigh_G(v^{(1)},\dots,v^{(r)})$.

The hypercube on $2^n$ vertices $\{-1,1\}^n$, viewed as a bipartite graph, will be denoted by $Q_n$.

\medskip

We will need the following
standard concentration inequality for independent Bernoulli variables.
\begin{lemma}[Chernoff]\label{kjfaoifboqfuifhbluhb}
Let $b_1,\dots,b_n$ be i.i.d.\ (independent and identically distributed) Ber\-noulli($p$) random variables (with $p\in(0,1)$). Then
$$
\Prob\bigg\{\bigg|\sum_{i=1}^n b_i-pn\bigg|\geq t\bigg\}\leq 2\exp\bigg(
-\frac{c_{\text{\tiny\ref{kjfaoifboqfuifhbluhb}}} t^2}{pn}\bigg),\quad t\in(0,pn],
$$
where $c_{\text{\tiny\ref{kjfaoifboqfuifhbluhb}}}\in(0,1]$ is a universal constant.
\end{lemma}

\section{The dependent random choice}\label{s:dep}

In this section, we revise the version of the dependent random choice method
outlined in the introduction in {\bf(I)--(II)}.
The material of this section, with some minor modifications, is taken from \cite{FS2009,FS2011}.
We provide the proofs for completeness.

\begin{lemma}\label{akejnfpeifjnpfijnqf}
Let $G' = (V^{up}_{G'},V^{down}_{G'},E_{G'})$ be a bipartite graph
of a density $\alpha\in(0,1]$; let $\beta\in(0,\alpha]$ and 
let $r,s\in \N$ be arbitrary parameters.
Let $X_1,X_2,\dots,X_s$ be i.i.d.\ uniform elements of $V^{down}_{G'}$.
Then
$$
\Exp\,|\cneigh_{G'}(X_1,X_2,\dots,X_s)|\geq \alpha^s\,|V^{up}_{G'}|,
$$
and
the expected number of ordered $r$--tuples of
elements of $\cneigh_{G'}(X_1,X_2,\dots,X_s)$ with at most $\beta^r|V^{down}_{G'}|$ common neighbors, is at most
$$
\beta^{rs}\,|V^{up}_{G'}|^r.
$$
\end{lemma}
\begin{proof}
Denote $A:=\cneigh_{G'}(X_1,X_2,\dots,X_s)$.
We have
\begin{align*}
\Exp\,|A|=\sum_{v\in V^{up}_{G'}}\Prob\{v\in A\}&
=\sum_{v\in V^{up}_{G'}}\Prob\{X_1\in\neigh_{G'}(v)\}^s\\
&=|V^{up}_{G'}|\,\sum_{v\in V^{up}_{G'}}\frac{1}{|V^{up}_{G'}|}
\bigg(\frac{|\neigh_{G'}(v)|}{|V^{down}_{G'}|}\bigg)^s\\
&\geq |V^{up}_{G'}|\,\bigg(\sum_{v\in V^{up}_{G'}}\frac{|\neigh_{G'}(v)|}{|V^{up}_{G'}|\,|V^{down}_{G'}|}\bigg)^s\\
&=|V^{up}_{G'}|\,\alpha^s.
\end{align*}
Further, for every $r$--tuple of distinct elements $y_1,\dots,y_r$ of $V^{up}_{G'}$, the probability of the event $\{y_1,\dots,y_r\}\subset A$
equals
$$
\Prob\big\{y_1,\dots,y_r\in \neigh_{G'}(X_1)\big\}^s
=\bigg(\frac{|\mbox{set of common neighbors of $y_1,\dots,y_r$}|}{|V^{down}_{G'}|}\bigg)^s.
$$
Thus, the expected number of ordered $r$--tuples in $A$ with at most $\beta^r |V^{down}_{G'}|$ common neighbors, is at most
$$
\beta^{rs}\,|V^{up}_{G'}|^r.
$$
\end{proof}

As a corollary, we have
\begin{cor}\label{aljfhbefwhbfoufbowfub}
For every $\varepsilon\in(0,1)$ there is
$n_{\text{\tiny\ref{aljfhbefwhbfoufbowfub}}}=n_{\text{\tiny\ref{aljfhbefwhbfoufbowfub}}}(\varepsilon)>0$
with the following property.
Let $G' = (V^{up}_{G'},V^{down}_{G'},E_{G'})$ be a bipartite graph
of density at least
$\alpha\in[\varepsilon,1]$, let $n\geq n_{\text{\tiny\ref{aljfhbefwhbfoufbowfub}}}$, and assume that
either
$$
|V^{up}_{G'}|\geq 2^{n+\varepsilon n},\quad |V^{down}_{G'}|\geq \frac{2^{n+\varepsilon n}}{\alpha^n}.
$$
or
$$
|V^{down}_{G'}|\geq 2^{n+\varepsilon n},\quad |V^{up}_{G'}|\geq \frac{2^{n+\varepsilon n}}{\alpha^n}.
$$
Then the hypercube $Q_n$ can be embedded into $G'$.
\end{cor}
\begin{proof}
Fix any $\varepsilon\in(0,1)$. We will assume that $n$ is large.
Let $G'$ be the graph satisfying the
above assumptions. 
We can suppose without loss of generality that
$|V^{up}_{G'}|\geq 2^{n+\varepsilon n}$, $|V^{down}_{G'}|\geq \frac{2^{n+\varepsilon n}}{\alpha^n}$.
Set
$$
s:=\bigg\lfloor\frac{\log(|V^{up}_{G'}|/2^n)}{\log(1/\alpha)}\bigg\rfloor,\quad
\beta:=\frac{2}{|V^{down}_{G'}|^{1/n}}.
$$
Observe that $\beta\leq 2^{-\varepsilon}\alpha$ and that
$$
\frac{|V^{up}_{G'}|}{2^n}\geq 
2^{\varepsilon n}= \varepsilon^{-\varepsilon n/\log_2(1/\varepsilon)}
\geq\frac{1}{\alpha^{\varepsilon n/\log_2(1/\varepsilon)}}
$$
implying
\begin{equation}\label{alkrhjbe43087aljkbf}
s\geq \lfloor \varepsilon n/\log_2(1/\varepsilon)\rfloor.
\end{equation}
Let $X_1,X_2,\dots,X_s$ be i.i.d.\ uniform elements of $V^{down}_{G'}$, and denote
$A:=\cneigh_{G'}(X_1,X_2,\dots,X_s)$.
In view of Lemma~\ref{akejnfpeifjnpfijnqf}, $\Exp\,|A|\geq \alpha^s\,|V^{up}_{G'}|$
whereas deterministically $|A|\leq |V^{up}_{G'}|$.
This implies that with probability at least $\frac{1}{2}\alpha^s$,
\begin{equation}\label{fvuachgvuyvqiyv}
|A|\geq \frac{1}{2}\alpha^s\,|V^{up}_{G'}|\geq 2^{n-1}.
\end{equation}
Combining this with the second assertion of the lemma, we get that with a positive probability the set $A$
satisfies \eqref{fvuachgvuyvqiyv}, and
the number of ordered $n$--tuples of
elements of $A$ with at most $\beta^n|V^{down}_{G'}|$ common neighbors is at most
$4\alpha^{-s}\beta^{ns}\,|V^{up}_{G'}|^n$.
We fix such realization of $A$ for the rest of the proof.

Let $(Y_1,\dots,Y_n)$ be a uniform random ordered $n$--tuple of distinct elements in $A$.
In view of the above, the probability that $(Y_1,\dots,Y_n)$ have less than $\beta^n|V^{down}_{G'}|$ common neighbors,
is at most
$$
\frac{4\alpha^{-s}\beta^{ns}\,|V^{up}_{G'}|^n}{|A|\,(|A|-1)\,\cdots\,(|A|-n+1)}
\leq\frac{4\alpha^{-s}\beta^{ns}\,|V^{up}_{G'}|^n}{\big(\frac{1}{2}\alpha^s\,|V^{up}_{G'}|-n\big)^n}
\leq 8\cdot 2^n\,\bigg(\frac{1}{\alpha}\Big(\frac{\beta}{\alpha}\Big)^{n}\bigg)^s,
$$
where in the last inequality we used \eqref{fvuachgvuyvqiyv} and our assumption that $n$ is sufficiently large.
Further, we have $\frac{1}{\alpha}\big(\frac{\beta}{\alpha}\big)^{n}\leq \frac{1}{\varepsilon}\,2^{-\varepsilon n}<1$,
and, in view of  \eqref{alkrhjbe43087aljkbf},
\begin{equation}\label{alfjhblfhjebolhbolu}
8\cdot 2^n\,\bigg(\frac{1}{\alpha}\Big(\frac{\beta}{\alpha}\Big)^{n}\bigg)^s
\leq 8\cdot 2^n\bigg(\frac{1}{\varepsilon}\,2^{-\varepsilon n}\bigg)^{\lfloor \varepsilon n/\log_2(1/\varepsilon)\rfloor}
<2^{-n+1},
\end{equation}
whenever $n$ is greater than a large constant multiple of $\varepsilon^{-2}\,\log_2(1/\varepsilon)$.
Set
$$\mathcal T:=\big\{v\in \{-1,1\}^n:\;\mbox{vector $v$ has an odd number of $-1$'s}\big\},$$
and let $f(\mathcal T)$ be the uniform random $2^{n-1}$--tuple of distinct elements in $A$.
Then, by \eqref{alfjhblfhjebolhbolu}, with a positive probability for every $v\in\{-1,1\}^n\setminus \mathcal T$,
the set of vertices $f(\neigh_{Q_n}(v))$ has at least $\beta^n|V^{down}_{G'}|\geq 2^{n-1}$
common neighbors in $G'$. It follows that with a positive probability the mapping $f:\mathcal T\to A$
can be extended to an embedding of $Q_n$ into $G'$. 
\end{proof}

\medskip

The last statement immediately implies the bound $r(Q_n)\leq 2^{2n+o(n)}$ via the standard
construction which we already mentioned in the introduction; see Remark~\ref{aljehbfoweufb} there.

\section{Condensed common neighborhoods}\label{s:disp}

The goal of this section is to investigate graph embeddings using dependent random choice
in the setting when common neighborhoods of subsets of vertices of the host graph tend to have small overlaps. 

Lemma~\ref{akejnfpeifjnpfijnqf} gives a probabilistic description of the set of common neighbors
of i.i.d.\ uniform random vertices in a bipartite graph.
We start by 
considering a de-randomization of the lemma similar to the one in the proof
of Corollary~\ref{aljfhbefwhbfoufbowfub}.
The de-randomization is accomplished by an application of Markov's inequality and
a union bound estimate. 
For convenience, we introduce a technical definition of a {\it standard}
vertex pair which groups together the properties useful for us.

\begin{defi}\label{pidufqpifunfpinfpiqnpi}
Let $G' = (V^{up}_{G'},V^{down}_{G'},E_{G'})$ be a bipartite graph
of density at least $\alpha\in(0,1]$, and let $\alpha_0\in(0,\alpha)$, $\mu\in(0,\alpha_0/2]$, $r\in \N$, and $K>0$.
An ordered pair $(v_1,v_2)$ of [not necessarily distinct]
vertices in $V^{down}_{G'}$ is
{\bf $(\alpha_0,\alpha,\mu,r,K)$--standard} if the following is true:
\begin{itemize}
\item The number of common neighbors of $v_1,v_2$ in $G'$ is at least $(1-\mu)\alpha^2\,|V^{up}_{G'}|$;
\item For every $1\leq k\leq |\cneigh_{G'}(v_1,v_2)|$, $m\geq1$, and for every finite collection
$(I_j)_{j=1}^m$ of subsets of $[k]$ satisfying $|I_j|=r$, $1\leq j\leq m$, we have:
if $(Y_i)_{i=1}^k$
is a random $k$--tuple of vertices in $\cneigh_{G'}(v_1,v_2)$ uniformly distributed
on the set of $k$--tuples of distinct vertices in $\cneigh_{G'}(v_1,v_2)$ then
with probability at least $1/2$,
\begin{align*}
&\big|\big\{
1\leq j\leq m:\;|\cneigh_{G'}(Y_i,\,i\in I_j)|\leq \beta^r|V^{down}_{G'}|
\big\}\big|\\
&\hspace{2cm}\leq m\cdot K\beta^{2r}\,\alpha^{-2r}\quad \mbox{for every}\quad\beta\in [\alpha_0,\alpha].
\end{align*}
\end{itemize}
\end{defi}

\begin{Remark}
Regarding the first part of the above definition, note that, according to Lem\-ma~\ref{akejnfpeifjnpfijnqf}, the expected number of common neighbors
for a random pair of vertices $v_1,v_2\in V^{down}_{G'}$ is at least $\alpha^2\,|V^{up}_{G'}|$.
The extra multiple $1-\mu$ controls acceptable deviation from the expected cardinality.
\end{Remark}

\begin{Remark}
The second condition can be roughly interpreted as ``the size of a common neighborhood
$\cneigh_{G'}(y_1,\dots,y_r)$, for $y_1,\dots,y_r\in \cneigh_{G'}(v_1,v_2)$, is typically 
of order at least $\Omega(K^{-1/2}\alpha^{r}|V^{down}_{G'}|)$''. Again, according to Lemma~\ref{akejnfpeifjnpfijnqf}, for a random $r$--tuple
of elements in $V^{up}_{G'}$ their common neighborhood in $V^{down}_{G'}$ has expected cardinality at least $\alpha^{r}|V^{down}_{G'}|$.
The rather complicated second part of the
definition involving the collections $(I_j)_{j=1}^m$ will turn out convenient when dealing with a standard vertex pair in the proof of Lemma~\ref{fljknfefjwfpiwjff}.
Similarly to the first condition, we introduce a parameter ($K$) to control the deviation of the actual set size from the expected cardinality in uniform random setting.
\end{Remark}

\begin{lemma}[Existence of standard pairs of vertices]\label{dfalnkfjnefwehfbqlwuh}
For every $\alpha_0\in(0,1)$ and $\mu\in(0,\alpha_0/2]$ there is
$C_{\text{\tiny\ref{dfalnkfjnefwehfbqlwuh}}}
=C_{\text{\tiny\ref{dfalnkfjnefwehfbqlwuh}}}(\alpha_0,\mu)\geq 1$ with the following property.
Let $\alpha\in(\alpha_0,1]$, $r\geq 2$, and let $G' = (V^{up}_{G'},V^{down}_{G'},E_{G'})$ be a bipartite graph
of density at least $\alpha$. Assume that
\begin{equation}\label{afnafkjfajhbfgavfg}
\alpha^2\,|V^{up}_{G'}|\geq r^2.
\end{equation}
Then
there exists an $(\alpha_0,\alpha,\mu,r,C_{\text{\tiny\ref{dfalnkfjnefwehfbqlwuh}}} r^3)$--standard
ordered pair $(v_1,v_2)$ in $V^{down}_{G'}$.
\end{lemma}
\begin{proof}
Set
$$
\tilde\delta:=\frac{\mu}{r}.
$$
Let $X_1,X_2$ be i.i.d.\ uniform random elements of $V^{down}_{G'}$, and set $A:=\cneigh_{G'}(X_1,X_2)$.
According to Lemma~\ref{akejnfpeifjnpfijnqf}, we have
$$
\Exp\,|A|\geq \alpha^2\,|V^{up}_{G'}|,
$$
whereas at the same time clearly $|A|\leq |V^{up}_{G'}|$ deterministically.
Thus,
$$
|V^{up}_{G'}|\,\Prob\big\{|A|\geq (1-\tilde\delta)\alpha^2\,|V^{up}_{G'}|\big\}
+(1-\tilde\delta)\alpha^2\,|V^{up}_{G'}|\,\big(1-\Prob\big\{|A|\geq (1-\tilde\delta)\alpha^2\,|V^{up}_{G'}|\big\}\big)
\geq \alpha^2\,|V^{up}_{G'}|,
$$
implying
$$
\Prob\big\{|A|\geq (1-\tilde\delta)\alpha^2\,|V^{up}_{G'}|\big\}
\geq \tilde\delta\alpha^2.
$$
Further, let $\beta_\ell:=\alpha\big(1-\frac{\alpha-\alpha_0}{\alpha}\frac{\ell}{r}\big)$, $\ell=0,\dots,r-1$, and let 
$L:=2\big(\tilde\delta\alpha^2\big)^{-1}$.
Denote by $\Event$ the event that
for every $\ell=0,\dots,r-1$
the number of ordered $r$--tuples of
elements of $A$ with at most $\beta_\ell^r|V^{down}_{G'}|$ common neighbors, is at most
$$
Lr\beta_\ell^{2r}\,|V^{up}_{G'}|^r.
$$
Applying Lemma~\ref{akejnfpeifjnpfijnqf} together with Markov's inequality,
we get that the probability of the intersection of events
$\Event\cap \{|A|\geq (1-\tilde\delta)\alpha^2\,|V^{up}_{G'}|\}$ is at least $\tilde\delta\alpha^2-\frac{1}{L}>0$.

It remains to check that, conditioned on $\Event\cap \{|A|\geq (1-\tilde\delta)\alpha^2\,|V^{up}_{G'}|\}$,
the pair $(X_1,X_2)$ is $(\alpha_0,\alpha,\mu,r,C'\,Lr^2)$--standard
for some $C'=C'(\alpha_0)>0$.
For the rest of the proof, we fix a realization of $X_1,X_2$ from $\Event\cap \{|A|\geq (1-\tilde\delta)\alpha^2\,|V^{up}_{G'}|\}$.
Pick any $1\leq k\leq |A|$, $m\geq1$, and any finite collection
$(I_j)_{j=1}^m$ of subsets of $[k]$ satisfying $|I_j|= r$, $1\leq j\leq m$.
Further, let $(Y_i)_{i=1}^k$
be a random $k$--tuple of vertices in $A$ uniformly distributed
on the set of $k$--tuples of distinct vertices in $A$.
In view of our conditions on $X_1,X_2$, \eqref{afnafkjfajhbfgavfg}, and the definition of $\tilde\delta$,
we have for every $j\leq m$,
$$
\Prob\big\{|\cneigh_{G'}(Y_i,\,i\in I_j)|\leq \beta_\ell^r|V^{down}_{G'}|\big\}
\leq \frac{Lr\beta_\ell^{2r}\,|V^{up}_{G'}|^r}{|A|\cdot (|A|-1)\cdot\dots\cdot (|A|-r+1)}
\leq \tilde C Lr\beta_\ell^{2r}\,\alpha^{-2r}
$$
for some universal constant $\tilde C>0$,
whence
$$
\Exp\,\big|\big\{
1\leq j\leq m:\;|\cneigh_{G'}(Y_i,\,i\in I_j)|\leq \beta_\ell^r|V^{down}_{G'}|
\big\}\big|\leq \tilde C m\cdot Lr\beta_\ell^{2r}\,\alpha^{-2r},\quad 0\leq \ell\leq r-1.
$$
Applying Markov's inequality (this time with respect to the randomness of $Y_i$'s), we get
that with [conditional] probability at least $1/2$,
\begin{equation}\label{fajnfpeifunpfunfpiun}
\big|\big\{
1\leq j\leq m:\;|\cneigh_{G'}(Y_i,\,i\in I_j)|\leq \beta_\ell^r|V^{down}_{G'}|
\big\}\big|\leq m\cdot 2\tilde C Lr^2\beta_\ell^{2r}\,\alpha^{-2r},\quad 0\leq \ell\leq r-1.
\end{equation}
Finally, assuming that \eqref{fajnfpeifunpfunfpiun} holds,
take any $\beta\in[\alpha_0,\alpha]$, and let $\ell\in\{0,\dots,r-1\}$
be the largest index such that $\beta_\ell\geq\beta$.
Note that $\beta_\ell\leq \beta+\frac{\alpha-\alpha_0}{r}
\leq \big(1+\frac{\alpha}{\alpha_0}\frac{1}{r}\big)\beta$.
Then
\begin{align*}
\big|\big\{&
1\leq j\leq m:\;|\cneigh_{G'}(Y_i,\,i\in I_j)|\leq \beta^r|V^{down}_{G'}|
\big\}\big|\\
&\leq 
\big|\big\{
1\leq j\leq m:\;|\cneigh_{G'}(Y_i,\,i\in I_j)|\leq \beta_\ell^r|V^{down}_{G'}|
\big\}\big|\\
&\leq m\cdot 2\tilde C Lr^2\beta_\ell^{2r}\,\alpha^{-2r}\\
&\leq m\cdot C' Lr^2\beta^{2r}\,\alpha^{-2r},
\end{align*}
for some $C'=C'(\alpha_0)>0$.
The result follows.
\end{proof}

Next, we discuss the main notion of this section.
\begin{defi}
Let $G' = (V^{up}_{G'},V^{down}_{G'},E_{G'})$ be a bipartite graph,
and let $r\in\N$, $M>0$, $p\in[0,1]$ be parameters.
Further, let $(v_1,v_2)$ be an ordered pair of vertices in $V^{down}_{G'}$ with a non-empty
set of common neighbors.
We say that the collection
$\big\{\cneigh_{G'}(y_1,\dots,y_r):\,(y_1,\dots,y_r)\in \cneigh_{G'}(v_1,v_2)^r\big\}$
is {\bf $(p,M)$--condensed} if, letting $Y_1,\dots,Y_r$, $\tilde Y_1,\dots,\tilde Y_r$
be i.i.d.\ uniform random elements of $\cneigh_{G'}(v_1,v_2)$, we have
$$
\Prob\big\{\big|\cneigh_{G'}(Y_1,\dots,Y_r)\cap \cneigh_{G'}(\tilde Y_1,\dots,\tilde Y_r)\big|
\geq M\big\}\geq p.
$$
\end{defi}

\begin{Remark}
The above definition will be applied to standard pairs of vertices $(v_1,v_2)$,
i.e in the setting when a typical common neighborhood $\cneigh_{G'}(y_1,\dots,y_r)$ has size
of order at least $\Omega(K^{-1/2}\alpha^r |V^{down}_{G'}|)$ (for an appropriate choice of the parameter $K$), where $\alpha$ is the edge density of $G'$. 
For $M$ much less than $\alpha^r |V^{down}_{G'}|$ and for $p=o(1)$, the assertion that the collection
$\big\{\cneigh_{G'}(y_1,\dots,y_r):\,(y_1,\dots,y_r)\in \cneigh_{G'}(v_1,v_2)^r\big\}$ is {\it not} $(p,M)$--condensed implies
that the neighborhoods $\cneigh_{G'}(y_1,\dots,y_r)$
typically have small overlaps.
\end{Remark}

\begin{lemma}[Embedding into a graph comprising a non-condensed set of common neighborhoods]\label{fljknfefjwfpiwjff}
Let $G' = (V^{up}_{G'},V^{down}_{G'},E_{G'})$ be a bipartite graph
of density at least $\alpha\in(0,1]$, and assume that parameters $0<\alpha_0<\alpha$, $\mu\in(0,\alpha_0/2]$,
$r\geq 3$ and $m\in\N$
satisfy $\alpha^2\,|V^{up}_{G'}|\geq 2\max(m,2 \alpha^r|V^{down}_{G'}|)$,
$\alpha_0^r|V^{down}_{G'}|\leq 1$, and
\begin{equation}\label{ahgvjhgvajhcakhc}
\frac{1}{4}\frac{\alpha^{2r}|V^{down}_{G'}|}{\max(m,2 \alpha^r|V^{down}_{G'}|)
\cdot C_{\text{\tiny\ref{dfalnkfjnefwehfbqlwuh}}}\,r^3}
\geq \alpha_0^r,
\end{equation}
where $C_{\text{\tiny\ref{dfalnkfjnefwehfbqlwuh}}}=C_{\text{\tiny\ref{dfalnkfjnefwehfbqlwuh}}}(\alpha_0,\mu)$
is taken from Lemma~\ref{dfalnkfjnefwehfbqlwuh}.
Assume further that 
$(v_1,v_2)$ is an ordered pair of vertices in $V^{down}_{G'}$
which is $(\alpha_0,\alpha,\mu,r,C_{\text{\tiny\ref{dfalnkfjnefwehfbqlwuh}}}\, r^3)$--standard.
Assume that the collection
$\big\{\cneigh_{G'}(y_1,\dots,y_r):\,(y_1,\dots,y_r)\in \cneigh_{G'}(v_1,v_2)^r\big\}$
is {\bf not} $(p,M)$--condensed where the parameters $p\in[0,1]$ and $M>0$ satisfy
\begin{equation}\label{ahfcdquyevwucytqvq}
r^2\leq p\cdot \alpha^r|V^{down}_{G'}|,\quad\quad
\sqrt{p}\leq \frac{c_{\text{\tiny\ref{kjfaoifboqfuifhbluhb}}}\,\alpha^{2r}\,|V^{down}_{G'}|^2}
{16\cdot 3^7 C_{\text{\tiny\ref{dfalnkfjnefwehfbqlwuh}}}\, r^3 \max(m^2 ,4 \alpha^{2r}|V^{down}_{G'}|^2) 
\cdot 20\log |V^{down}_{G'}|}
\end{equation}
and
\begin{equation}\label{akhgdvcagcajgcvkvh}
1\leq M\leq \frac{c_{\text{\tiny\ref{kjfaoifboqfuifhbluhb}}}^2\,\alpha^{4r}|V^{down}_{G'}|^4}{2^9 C_{\text{\tiny\ref{dfalnkfjnefwehfbqlwuh}}}^2\,\max(m^3,8 \alpha^{3r}|V^{down}_{G'}|^3) \,r^6 
\cdot 400\log^2 |V^{down}_{G'}|},
\end{equation}
where the constant $c_{\text{\tiny\ref{kjfaoifboqfuifhbluhb}}}$ is taken from Lemma~\ref{kjfaoifboqfuifhbluhb}.
Further, let $H=(V^{up}_{H},V^{down}_{H},E_{H})$ be an $r$--regular bipartite graph on $m+m$ vertices. 
Then $H$ can be embedded into $G'$.
\end{lemma}
\begin{Remark}
Observe that the lemma does not require any structural assumptions on the graph $H$ except for the
regularity and bounds on the size of the vertex set.
Moreover, one may consider versions of this lemma which operate under the only assumptions
on the vertex set cardinality and the maximum degree, without the regularity requirement.
We omit the discussion of such generalizations in order not to complicate the exposition further.
\end{Remark}

Before providing a proof of the lemma, let us discuss our construction of the graph embedding
rather informally. Let $(v_1,v_2)$ be a standard pair of vertices in $V^{down}_{G'}$ satisfying the above assumptions.
For convenience, we label vertices of $H$ as $t_1^{up},\dots,t_m^{up}$, $t_1^{down},\dots,t_m^{down}$,
and let $I_j$, $1\leq j\leq m$,
be $r$--subsets of $[m]$ corresponding to neighbors of vertices in $V^{down}_{H}$, i.e.\
$\neigh_H(t_j^{down})=\{t_i^{up},\,i\in I_j\}$ for every $j\leq m$.
The embedding of $H$ is accomplished by mapping the vertices $V^{up}_{H}$
into $V^{up}_{G'}$ using the first moment argument,
and then embedding $V^{down}_{H}$ into $V^{down}_{G'}$ one
vertex at a time via a combination of deterministic and probabilistic reasoning.

The first moment method is applied to produce an $m$--tuple of distinct vertices $(y_1,y_2,\dots,y_m)$
in $\cneigh_{G'}(v_1,v_2)$, such that
\begin{itemize}
\item The number of indices $j\leq m$ with $|\cneigh_{G'}(y_i,\,i\in I_j)|$
much smaller than $\alpha^r|V^{down}_{G'}|$, is much smaller than $m$, and
\item The pairwise intersections $\cneigh_{G'}(y_i,\,i\in I_{j_1})\cap \cneigh_{G'}(y_i,\,i\in I_{j_2})$
have size less than $M$ for
a vast majority of pairs of indices $(j_1,j_2)\in[m]^2$.
\end{itemize}
We refer to the proof below for quantitative bounds.
The vertex subset $V^{up}_{H}$ is then mapped to $y_1,y_2,\dots,y_m$.
This provides a satisfactory starting point for the embedding
since, by the above conditions, the common neighborhoods of $\{y_i,\,i\in I_j\}$, $j\leq m$,
tend to have relatively large cardinalities and small pairwise overlaps.
At this stage, the goal is to map $V^{down}_{H}$ into a collection of distinct
vertices $f(t_1^{down}),\dots,f(t_m^{down})$ from $V^{down}_{G'}$ such that $f(t_j^{down})$
is contained in the common neighborhood of $\{y_i,\,i\in I_j\}$, for every $j\leq m$.
We split the index set $[m]$ into three subsets $Q,W,R$
according to properties of the common neighborhoods $\cneigh_{G'}(y_i,\,i\in I_j)$
as follows:
\begin{itemize}
\item $Q$ is a set of indices $j$ such 
that $|\cneigh_{G'}(y_i,\,i\in I_j)|$ is very small;
\item $W$ is the subset of all $j\in [m]\setminus Q$ such that
the common neighborhood $\cneigh_{G'}(y_i,\,i\in I_j)$
has large overlaps with many other sets $\cneigh_{G'}(y_i,\,i\in I_{\tilde j})$,
$\tilde j\neq j$;
\item $R$ is the complement of $Q$ and $W$ in $[m]$, the set of ``regular'' indices.
\end{itemize}
Both $Q$ and $W$ are small (and possibly empty), and $t_j^{down}$, $j\in Q\cup W$,
are embedded in the first place using a deterministic argument.
Strong upper bounds on sizes of $Q$ and $W$ guarantee that the embedding does not fail at this point.
The embedding of $t_j^{down}$, $j\in R$, is then accomplished via a randomized construction.

We now turn to the rigorous argument.

\begin{proof}[Proof of Lemma~\ref{fljknfefjwfpiwjff}]
For better readability, we split the proof into blocks.

\medskip

{\bf An assumption on $m$.}
We claim that without loss of generality we can assume that the parameter $m$ satisfies
\begin{equation}\label{iyevfduqtfdcuqfc}
m\geq \alpha^r|V^{down}_{G'}|.
\end{equation}
Indeed, suppose that the lemma is proved under the extra assumption \eqref{iyevfduqtfdcuqfc}. Let $\tilde m< \alpha^r|V^{down}_{G'}|$ be any positive integer, let $\tilde H$
be a bipartite $r$--regular graph on $\tilde m+\tilde m$ vertices, and suppose that the parameters $\alpha,\alpha_0,\mu,r,p,M$ and an ordered pair $(v_1,v_2)$
all satisfy the assumptions of the lemma, with $m$ replaced with $\tilde m$ in \eqref{ahgvjhgvajhcakhc},~\eqref{ahfcdquyevwucytqvq}, and~\eqref{akhgdvcagcajgcvkvh}.
We want to show that $\tilde H$ can be embedded into $G'$.
Define $m$ as the smallest integer multiple of $\tilde m$ satisfying \eqref{iyevfduqtfdcuqfc} and observe that $m\leq 2\alpha^r|V^{down}_{G'}|$ implying that
$\max(m,2 \alpha^r|V^{down}_{G'}|)=\max(\tilde m,2 \alpha^r|V^{down}_{G'}|)=2 \alpha^r|V^{down}_{G'}|$ and hence $m$ satisfies
\eqref{ahgvjhgvajhcakhc},~\eqref{ahfcdquyevwucytqvq}, and~\eqref{akhgdvcagcajgcvkvh}. Define a bipartite graph $H$ on $m+m$ vertices as the disjoint union of $m/\tilde m$ copies of the graph $\tilde H$.
By our assumption, $H$ can be embedded into $G'$, and therefore the same is true for $\tilde H$. This proves the claim.
For the rest of the proof, the parameter $m$ is supposed to satisfy \eqref{iyevfduqtfdcuqfc}.

\medskip

{\bf Choosing an $m$--tuple of vertices in $V^{up}_{G'}$.}
Let $(t_1^{up},\dots,t_m^{up})$ and
$(t_1^{down},\dots,t_m^{down})$ be the vertices of $H$ from $V^{up}_{H}$ and $V^{down}_{H}$,
respectively, 
ordered arbitrarily.
As in the proof outline above, for every $1\leq j\leq m$, we let $I_j$ be the collection of indices in $[m]$
such that $\neigh_H(t_j^{down})=\{t_i^{up},\,i\in I_j\}$
(we observe that, in view of $r$--regularity of $H$,
the number of all ordered pairs of indices $(j_1,j_2)\in[m]^2$ such that $I_{j_1}
\cap I_{j_2}=\emptyset$, is at least 
$m\cdot (m-r^2)$).

Let $(Y_i)_{i=1}^m$
be a random $m$--tuple of vertices in $\cneigh_{G'}(v_1,v_2)$ uniformly distributed
on the set of $m$--tuples of distinct vertices in $\cneigh_{G'}(v_1,v_2)$ (the condition
that $(v_1,v_2)$ is standard and our assumption on $|V^{up}_{G'}|$ imply that $|\cneigh_{G'}(v_1,v_2)|\geq m$ so that
$(Y_i)_{i=1}^m$ are well defined).
In view of the definition of $(p,M)$--condensation combined with the last observation, we have
$$
\Exp\,\Big|\Big\{(j_1,j_2)\in[m]^2:\;
\big|\cneigh_{G'}(Y_i,\,i\in I_{j_1})\cap \cneigh_{G'}(Y_i,\,i\in I_{j_2})\big| \geq M
\Big\}\Big|
\leq m\cdot r^2+\frac{p\cdot m^2}{\varrho},
$$
where $\varrho\in(0,1)$ is the probability that $2r$ i.i.d.\ uniform random elements of $\cneigh_{G'}(v_1,v_2)$ are all distinct.
Since $|\cneigh_{G'}(v_1,v_2)|\geq m\geq r^2$ in view of the first inequality in \eqref{ahfcdquyevwucytqvq} and \eqref{iyevfduqtfdcuqfc},
we have $m\cdot r^2\leq p\cdot m^2$ and
$$
\varrho\geq \bigg(1-\frac{2}{r}\bigg)^{2r}\geq 3^{-6}.
$$
Markov's inequality and the definition of a standard vertex pair then imply
that with a positive probability the collection $(Y_i)_{i=1}^m$ satisfies all of the following:
\begin{itemize}
\item[(a)]
$
\big|\big\{
1\leq j\leq m:\;|\cneigh_{G'}(Y_i,\,i\in I_j)|\leq \beta^r|V^{down}_{G'}|
\big\}\big|\leq m\cdot C_{\text{\tiny\ref{dfalnkfjnefwehfbqlwuh}}}\,r^3\beta^{2r}\,\alpha^{-2r}\; \mbox{for every }\beta\in [\alpha_0,\alpha];
$
\item[(b)] $\big|\cneigh_{G'}(Y_i,\,i\in I_{j_1})\cap \cneigh_{G'}(Y_i,\,i\in I_{j_2})\big| \geq M$ for
at most $3^7 p\cdot m^2$ pairs of indices $(j_1,j_2)\in[m]^2$.
\end{itemize}
For the rest of the proof, we fix a realization $(y_1,\dots,y_m)$ of $(Y_i)_{i=1}^m$ satisfying the above
conditions.
We shall construct an embedding $f$ of $H$ into $G$ which maps each
$t_i^{up}$ into $y_i$, $1\leq i\leq m$.

\medskip

{\bf Partitioning the set of indices $[m]$.}
Define $\beta_0$ via the relation
$$
\beta_0^{r}:=\frac{1}{4}\frac{\alpha^{2r}|V^{down}_{G'}|}{m\cdot C_{\text{\tiny\ref{dfalnkfjnefwehfbqlwuh}}}\,r^3},
$$
and observe that in view of
\eqref{ahgvjhgvajhcakhc} and \eqref{iyevfduqtfdcuqfc}, $\alpha_0\leq\beta_0\leq\alpha$, and hence, by the condition (a),
$$
\big|\big\{
1\leq j\leq m:\;|\cneigh_{G'}(y_i,\,i\in I_j)|\leq \beta_0^r|V^{down}_{G'}|
\big\}\big|\leq m\cdot C_{\text{\tiny\ref{dfalnkfjnefwehfbqlwuh}}}\,r^3\beta_0^{2r}\,\alpha^{-2r}.
$$
Let $Q$ be the set of all indices $1\leq j\leq m$ with
$|\cneigh_{G'}(y_i,\,i\in I_j)|\leq \beta_0^r|V^{down}_{G'}|$,
so that
\begin{equation}\label{fakjenfoifufiqjfqf}
|Q|\leq 
m\cdot C_{\text{\tiny\ref{dfalnkfjnefwehfbqlwuh}}}\,
r^3\beta_0^{2r}\,\alpha^{-2r}
=
\frac{1}{16}\frac{\alpha^{2r} |V^{down}_{G'}|^2}{m\cdot C_{\text{\tiny\ref{dfalnkfjnefwehfbqlwuh}}}\,r^3}
\leq \alpha^r\,|V^{down}_{G'}|,
\end{equation}
where the last inequality follows from \eqref{iyevfduqtfdcuqfc}.
Assume that
$(q_s)_{s=1}^{|Q|}$ is an ordering of $Q$ such that
$$
|\cneigh_{G'}(y_i,\,i\in I_{q_s})|\leq |\cneigh_{G'}(y_i,\,i\in I_{q_{s+1}})|,\quad 1\leq s<|Q|.
$$
Further, let $W$ be the set of all indices $1\leq j\leq m$ such that
\begin{itemize}
\item $|\cneigh_{G'}(y_i,\,i\in I_j)|> \beta_0^r|V^{down}_{G'}|$ and
\item $\big|\cneigh_{G'}(y_i,\,i\in I_{j})\cap \cneigh_{G'}(y_i,\,i\in I_{\tilde j})\big| \geq M$
for at least $\sqrt{p}\,m$ indices $\tilde j\in[m]$
\end{itemize}
(note that in view of the condition (b) above, $|W|\leq 3^7\sqrt{p}\,m$),
and let $(w_s)_{s=1}^{|W|}$ be an arbitrary ordering of the vertices from $W$.
Finally, we let $R:=[m]\setminus (Q\cup W)$, i.e $R$ is the set of indices 
$1\leq j\leq m$ such that
\begin{itemize}
\item $|\cneigh_{G'}(y_i,\,i\in I_j)|> \beta_0^r|V^{down}_{G'}|$ and
\item $\big|\cneigh_{G'}(y_i,\,i\in I_{j})\cap \cneigh_{G'}(y_i,\,i\in I_{\tilde j})\big| \geq M$
for less than $\sqrt{p}\,m$ indices $\tilde j\in[m]$.
\end{itemize}
Similarly, let $(r_s)_{s=1}^{|R|}$
be an arbitrary ordering of the vertices in $R$.

\medskip

{\bf A deterministic embedding of $t_j^{down}$, $j\in Q\cup W$.}
We define $f(t_j^{down})$, $j\in Q\cup W$ via a simple iterative procedure comprised of $|Q|+|W|$ steps.
A $s$--th step,
\begin{itemize}
\item If $1\leq s\leq |Q|$ then we let $f(t_{q_s}^{down})$ to be any point
in 
$$\cneigh_{G'}(y_i,\,i\in I_{q_s})\setminus\{f(t_{q_1}^{down}),\dots,f(t_{q_{s-1}}^{down})\};$$
\item If $|Q|+1\leq s\leq |Q|+|W|$ then we define $f(t_{w_{s-|Q|}}^{down})$
as an arbitrary point in
$$
\cneigh_{G'}(y_i,\,i\in I_{w_{s-|Q|}})\setminus\{f(t_{q_1}^{down}),\dots,f(t_{q_{|Q|}}^{down});
f(t_{w_1}^{down}),\dots,f(t_{w_{s-|Q|-1}}^{down})\}.
$$
\end{itemize}
To make sure that the above process does not fail, 
we need to verify that at each step $s$, $1\leq s\leq |Q|$, the set
$\cneigh_{G'}(y_i,\,i\in I_{q_s})\setminus\{f(t_{q_1}^{down}),\dots,f(t_{q_{s-1}}^{down})\}$
is necessarily non-empty, and similarly, $\cneigh_{G'}(y_i,\,i\in I_{w_{s-|Q|}})\setminus\{f(t_{q_1}^{down}),\dots,f(t_{q_{|Q|}}^{down});
f(t_{w_1}^{down}),\dots,f(t_{w_{s-|Q|-1}}^{down})\}\neq \emptyset$
for every $|Q|+1\leq s\leq |Q|+|W|$, regardless of the specific choices for $f(t_j^{down})$ at previous steps.

Observe that
the condition
$$
\cneigh_{G'}(y_i,\,i\in I_{q_s})\setminus\{f(t_{q_1}^{down}),\dots,f(t_{q_{s-1}}^{down})\}=\emptyset
\quad\mbox{for some $1\leq s\leq |Q|$},$$
together with our choice of the ordering $(q_s)_{s=1}^{|Q|}$, would imply that 
\begin{equation}\label{akjbefoiuyro3ufbljsefla}
|\cneigh_{G'}(y_i,\,i\in I_j)|\leq s\quad\mbox{for at least $s$ indices $j\in[m]$.}
\end{equation}
For $s<\alpha_0^r|V^{down}_{G'}|$, this would imply that the set
$\big\{
1\leq j\leq m:\;|\cneigh_{G'}(Y_i,\,i\in I_j)|\leq \alpha_0^r|V^{down}_{G'}|
\big\}$ is non-empty which would contradict the condition (a) and the inequality
$$
m\cdot C_{\text{\tiny\ref{dfalnkfjnefwehfbqlwuh}}}\,r^3\alpha_0^{2r}\,\alpha^{-2r}
\leq \frac{m\cdot C_{\text{\tiny\ref{dfalnkfjnefwehfbqlwuh}}}\,r^3\alpha_0^{r}}{\alpha^{2r}|V^{down}_{G'}|}<1,
$$
which follows from \eqref{ahgvjhgvajhcakhc} and the assumption $\alpha_0^r|V^{down}_{G'}|\leq 1$.
On the other hand, for $\alpha_0^r|V^{down}_{G'}|\leq s\leq |Q|$, \eqref{akjbefoiuyro3ufbljsefla}
combined with condition (a) and the cardinality estimate for $Q$ implies
$s\leq m\cdot C_{\text{\tiny\ref{dfalnkfjnefwehfbqlwuh}}}\,
r^3\,\frac{s^2}{|V^{down}_{G'}|^2}\,\alpha^{-2r}$, which, 
together with \eqref{fakjenfoifufiqjfqf}, yields
$$
m\cdot C_{\text{\tiny\ref{dfalnkfjnefwehfbqlwuh}}}\,r^3\beta_0^{2r}\,\alpha^{-2r}\geq
|Q|\geq \frac{\alpha^{2r} |V^{down}_{G'}|^2}{m\cdot C_{\text{\tiny\ref{dfalnkfjnefwehfbqlwuh}}}\,r^3},
$$
again leading to contradiction in view of the definition of $\beta_0$. 
Thus, the process defined above cannot fail at any step $1\leq s\leq |Q|$.

To verify that our embedding process does not fail at steps $s\in[|Q|+1,\dots,|Q|+|W|]$,
we note that, in view of the definition of $\beta_0$ and the second inequality in \eqref{ahfcdquyevwucytqvq},
\begin{equation}\label{aklfnfpiejnfpwifjnwp}
\beta_0^r|V^{down}_{G'}|\geq 2m\cdot C_{\text{\tiny\ref{dfalnkfjnefwehfbqlwuh}}}\,
r^3\beta_0^{2r}\,\alpha^{-2r}+2\cdot 3^7\sqrt{p}\,m\geq 2|Q|+2|W|.
\end{equation}

\medskip

{\bf A randomized embedding of vertices $t_j^{down}$, $j\in R$.}
To complete construction of our embedding $f$, it remains to define $f(t_{r_s})$, $s=1,\dots,|R|$.
Set
$$
h:=\Big\lceil\frac{10}{c_{\text{\tiny\ref{kjfaoifboqfuifhbluhb}}}}\,\log |V^{down}_{G'}|\Big\rceil.
$$

For every $j\in R$, let $Z_{j1},\dots,Z_{jh}$ be uniform random vertices in
$$
\cneigh_{G'}(y_i,\,i\in I_j)\setminus f(Q\cup W)
$$
(note that in view of \eqref{aklfnfpiejnfpwifjnwp}, the set difference is non-empty and, moreover,
$|\cneigh_{G'}(y_i,\,i\in I_j)\setminus f(Q\cup W)|\geq \frac{1}{2}|\cneigh_{G'}(y_i,\,i\in I_j)|$),
and assume that $Z_{j1},\dots,Z_{jh}$, $j\in R$, are mutually independent.
We then define $f(t_{r_s})$, $s=1,\dots,|R|$, as any $|R|$--tuple
of distinct vertices in $V^{down}_{G'}$
satisfying $f(t_{r_s})\in \{Z_{r_s 1},\dots,Z_{r_s h}\}$ for each $s=1,\dots,|R|$,
whenever such vertex assignment is possible, and declare failure otherwise.
To complete the proof, we must verify that this vertex assignment succeeds with a positive probability.
Note that a sufficient condition of success is
\begin{equation}\label{flaejhfbqwofuhbfljafla}
\{Z_{r_s 1},\dots,Z_{r_s h}\}\setminus \bigcup_{\tilde s=1}^{s-1}
\{Z_{r_{\tilde s} 1},\dots,Z_{r_{\tilde s} h}\}\neq\emptyset,\quad s=1,\dots,|R|.
\end{equation}
Pick any $s\in\{1,\dots,|R|\}$, and let $L_s$ be the collection of all indices $\tilde s\in\{1,\dots,s-1\}$
such that $\big|\cneigh_{G'}(y_i,\,i\in I_{r_s})\cap \cneigh_{G'}(y_i,\,i\in I_{r_{\tilde s}})\big| \geq M$.
In view of the definition of $R$ as the complement of $Q\cup W$, we have
$|L_s|<\sqrt{p}\,m$.
For every $\tilde s\in \{1,\dots,s-1\}\setminus L_s$, let $b_{\tilde s}$ be the indicator of the
event
$$
\big\{Z_{r_{\tilde s} 1},\dots,Z_{r_{\tilde s} h}\big\}\cap \cneigh_{G'}(y_i,\,i\in I_{r_s})\neq\emptyset.
$$
We have, in view of \eqref{aklfnfpiejnfpwifjnwp},
$$
\Prob\{b_{\tilde s}=1\}
\leq h\cdot \frac{2M}{\beta_0^r|V^{down}_{G'}|},
$$
whence, by Chernoff's inequality (Lemma~\ref{kjfaoifboqfuifhbluhb}),
$$
\Prob\bigg\{\sum_{\tilde s\in \{1,\dots,s-1\}\setminus L_s}b_{\tilde s}\geq 
\frac{4hm\cdot M}{\beta_0^r|V^{down}_{G'}|}
\bigg\}\leq 2\exp\bigg(-\frac{2c_{\text{\tiny\ref{kjfaoifboqfuifhbluhb}}} hm\cdot M}{\beta_0^r|V^{down}_{G'}|}\bigg).
$$
We conclude that with probability at least
$1-2\exp\big(-\frac{2c_{\text{\tiny\ref{kjfaoifboqfuifhbluhb}}} hm\cdot M}{\beta_0^r|V^{down}_{G'}|}\big)$,
we have
$$
\bigg|\cneigh_{G'}(y_i,\,i\in I_{r_s})\cap \bigcup_{\tilde s=1}^{s-1}
\{Z_{r_{\tilde s} 1},\dots,Z_{r_{\tilde s} h}\}\bigg|
\leq \frac{4h^2m\cdot M}{\beta_0^r|V^{down}_{G'}|}+h\sqrt{p}\,m\leq\frac{1}{4}\beta_0^r|V^{down}_{G'}|,
$$
where the last inequality follows as a combination of \eqref{ahfcdquyevwucytqvq} and \eqref{akhgdvcagcajgcvkvh}.
On the other hand, conditioned on the last estimate, we have that
$$
\{Z_{r_s 1},\dots,Z_{r_s h}\}\setminus \bigcup_{\tilde s=1}^{s-1}
\{Z_{r_{\tilde s} 1},\dots,Z_{r_{\tilde s} h}\}\neq\emptyset
$$
with [conditional] probability at least $1-(3/4)^{h}$.
Thus, \eqref{flaejhfbqwofuhbfljafla} holds with probability at least
$$
1-m\cdot\bigg((3/4)^{h}+
2\exp\bigg(-\frac{2c_{\text{\tiny\ref{kjfaoifboqfuifhbluhb}}}hm\cdot M}{\beta_0^r|V^{down}_{G'}|}\bigg)\bigg)>0,
$$
where in the last inequality we used the definition of $h$, \eqref{akhgdvcagcajgcvkvh}, and \eqref{iyevfduqtfdcuqfc}.
The proof is complete.
\end{proof}

\section{Embedding into block-structured graphs}\label{s:blockstructured}

In this section, we consider a special class of bipartite graphs
whose structure is similar to the graph $\Gamma$ from the introduction.
  
\begin{defi}
Let $G' = (V^{up}_{G'},V^{down}_{G'},E_{G'})$ be a bipartite graph.
We say that $G'$ is {\bf block-structured with parameters $(\delta,\gamma,k,g)$}
if there is a partition
$$
V^{down}_{G'}=\bigsqcup_{\ell=1}^k S^{down}_\ell
$$
of $V^{down}_{G'}$ into non-empty sets, and a collection of non-empty
subsets $S^{up}_\ell\subset V^{up}_{G'}$,
$\ell=1,\dots,k$, having the following properties:
\begin{itemize}
\item For each $\ell=1,\dots,k$, $|S^{down}_\ell|= g$;
\item For each $\ell=1,\dots,k$, $|S^{up}_\ell|\geq \gamma |V^{up}_{G'}|$;
\item For each $\ell=1,\dots,k$, vertices in $S^{down}_\ell$ are {\bf not} adjacent to any of the
vertices in $V^{up}_{G'}\setminus S^{up}_\ell$;
\item For each $\ell=1,\dots,k$, the density of the bipartite subgraph of $G'$ induced by the vertex subset
$S^{up}_\ell\sqcup S^{down}_\ell$, is at least $1-\delta$.
\end{itemize}
We will further say that collections of subsets $\big(S^{up}_\ell\big)_{\ell=1}^k$,
$\big(S^{down}_\ell\big)_{\ell=1}^k$ satisfying the above properties
are {\bf compatible} with the block-structured graph $G'$.
Note that the compatible collections may not be uniquely defined;
in what follows
for every block-structured bipartite graph (with given parameters)
we arbitrarily fix a pair of compatible collections $\big(S^{up}_\ell\big)_{\ell=1}^k$,
$\big(S^{down}_\ell\big)_{\ell=1}^k$, and refer to them as {\bf the} compatible subsets.
\end{defi}

\begin{Remark}
Note that from the above definition it follows that $|V^{down}_{G'}|=k\,g$.
\end{Remark}

\begin{defi}
Let $r,w,u\geq 1$, let $G'= (V^{up}_{G'},V^{down}_{G'},E_{G'})$
be a block-structured bipartite graph (with some parameters $(\delta,\gamma,k,g)$), and let
$\big(S^{up}_\ell\big)_{\ell=1}^k$,
$\big(S^{down}_\ell\big)_{\ell=1}^k$ be the corresponding
compatible sequences of subsets of vertices of $G'$.
Further, let $y_1,\dots,y_u\in V^{up}_{G'}$ (some of the vertices may repeat).
We define
$$\mathcal M_{G'}(r,w;y_1,\dots,y_u)$$
as the collection of all ordered $r$--tuples $(x_1,\dots,x_r)$ of elements of $V^{down}_{G'}$
satisfying the following conditions:
\begin{itemize}
\item The vertices $x_1,\dots,x_r$ are distinct;
\item There are distinct indices $1\leq \ell_0,\ell_1,\dots,\ell_w\leq k$ such that
$x_1,\dots,x_{r-w}\in S^{down}_{\ell_0}$ and for every $1\leq a\leq w$, $x_{r-w+a}\in S^{down}_{\ell_a}$;
\item $\{x_1,\dots,x_r\}\subset\cneigh_{G'}(y_1,\dots,y_u)$.
\end{itemize}
Note that in the notation for $\mathcal M_{G'}(\dots)$ we omit the parameters $\delta,\gamma,k,g$
for brevity.
\end{defi}

\begin{lemma}[Dependent random choice for block-structured graphs]\label{adkjnfpiffiqfijfakdfk}
Let $G'= (V^{up}_{G'},V^{down}_{G'},E_{G'})$ be a block-structured bipartite graph
with parameters $(\delta,\gamma,k,g)$, and let $r,w,u\geq 1$ with $g\geq r-w\geq 2$ and $k\geq w+1$.
Let $\big(S^{up}_\ell\big)_{\ell=1}^k$,
$\big(S^{down}_\ell\big)_{\ell=1}^k$ be the corresponding
compatible sequences of subsets of vertices of $G'$.
Further, assume that $Y_1,\dots,Y_u$ are i.i.d.\ uniform random elements of $V^{up}_{G'}$.
Then
\begin{align*}
\Prob\big\{&\big|\cneigh_{G'}(Y_1,\dots,Y_u)\cap S^{down}_{\ell}\big|\geq 
g\,(1-\delta)^{u+1}\\
&\mbox{ for at least $k\cdot \frac{\delta}{2}\big(\gamma(1-\delta)\big)^u$
indices $\ell$}\big\}\geq \frac{\delta}{2}\big(\gamma(1-\delta)\big)^u,
\end{align*}
and for any $s>0$,
the expected number of $r$--tuples $(x_1,\dots,x_r)\in \mathcal M_{G'}(r,w;Y_1,\dots,Y_u)$
with $|\cneigh_{G'}(x_1,\dots,x_r)|\leq s$ is bounded above by
$$
\bigg(\frac{s}{|V^{up}_{G'}|}\bigg)^u\,\frac{k!}{(k-w-1)!}\frac{g^w\,g!}{(g-r+w)!}.
$$
\end{lemma}
\begin{proof}
Fix for a moment any $1\leq \ell\leq k$. 
We have
\begin{equation}\label{lsjkfhb40387kjf}
\Prob\big\{Y_1,\dots,Y_u\in S^{up}_{\ell}\big\}\geq \gamma^u.
\end{equation}
On the other hand, conditioned on the event $\{Y_1,\dots,Y_u\in S^{up}_{\ell}\}$,
we get, by repeating the first part of the argument from the proof of Lemma~\ref{akejnfpeifjnpfijnqf},
$$
\Exp\,\Big[\big|\cneigh_{G'}(Y_1,\dots,Y_u)\cap S^{down}_{\ell}\big|\;\big|\;
Y_1,\dots,Y_u\in S^{up}_{\ell}\Big]\geq 
g\,(1-\delta)^u.
$$
Since the size of the intersection $\cneigh_{G'}(Y_1,\dots,Y_u)\cap S^{down}_{\ell}$
is deterministically bounded above by $g$, we obtain
\begin{align*}
g\,(1-\delta)^u&\leq
\Exp\,\Big[\big|\cneigh_{G'}(Y_1,\dots,Y_u)\cap S^{down}_{\ell}\big|\;\big|\;
Y_1,\dots,Y_u\in S^{up}_{\ell}\Big]\\
&\leq
g\,\Prob\big\{\big|\cneigh_{G'}(Y_1,\dots,Y_u)\cap S^{down}_{\ell}\big|\geq 
g\,(1-\delta)^{u+1}\;\big|\;Y_1,\dots,Y_u\in S^{up}_{\ell}\big\}\\
&+g\,(1-\delta)^{u+1}\,\big(1-\Prob\big\{\big|\cneigh_{G'}(Y_1,\dots,Y_u)\cap S^{down}_{\ell}\big|\geq
g\,(1-\delta)^{u+1}\;\big|\;Y_1,\dots,Y_u\in S^{up}_{\ell}\big\}\big),
\end{align*}
and hence
\begin{align*}
\Prob\big\{\big|\cneigh_{G'}(Y_1,\dots,Y_u)\cap S^{down}_{\ell}\big|\geq 
g\,(1-\delta)^{u+1}\;\big|\;Y_1,\dots,Y_u\in S^{up}_{\ell}\big\}&\geq
\frac{(1-\delta)^u-(1-\delta)^{u+1}}{1-(1-\delta)^{u+1}}\\
&\geq 
\delta(1-\delta)^u.
\end{align*}
That, together with \eqref{lsjkfhb40387kjf}, implies
$$
\Prob\big\{\big|\cneigh_{G'}(Y_1,\dots,Y_u)\cap S^{down}_{\ell}\big|\geq 
g\,(1-\delta)^{u+1}\big\}\geq \delta(1-\delta)^u\cdot \gamma^u.
$$
Since the last estimate is true for every $1\leq \ell\leq k$, we get
\begin{align*}
\bigg(k-k\cdot \frac{\delta}{2}\big(\gamma(1-\delta)\big)^u\bigg)\,\Prob\big\{&\big|\cneigh_{G'}(Y_1,\dots,Y_u)\cap S^{down}_{\ell}\big|\geq 
g\,(1-\delta)^{u+1}\\
&\mbox{ for at least $k\cdot \frac{\delta}{2}\big(\gamma(1-\delta)\big)^u$
indices $\ell$}\big\}
+k\cdot \frac{\delta}{2}\big(\gamma(1-\delta)\big)^u\\
&\hspace{-1cm}\geq \sum_{\ell=1}^k \Prob\big\{\big|\cneigh_{G'}(Y_1,\dots,Y_u)\cap S^{down}_{\ell}\big|\geq 
g\,(1-\delta)^{u+1}\big\}\geq k\, \delta(1-\delta)^u\cdot \gamma^u,
\end{align*}
and hence
\begin{align*}
\Prob\big\{&\big|\cneigh_{G'}(Y_1,\dots,Y_u)\cap S^{down}_{\ell}\big|\geq 
g\,(1-\delta)^{u+1}\\
&\mbox{ for at least $k\cdot \frac{\delta}{2}\big(\gamma(1-\delta)\big)^u$
indices $\ell$}\big\}
\geq \frac{\delta}{2}\big(\gamma(1-\delta)\big)^u.
\end{align*}

Further, take any ordered $r$--tuple $(x_1,\dots,x_r)$ of elements of $V^{down}_{G'}$
satisfying the following conditions:
\begin{itemize}
\item The vertices $x_1,\dots,x_r$ are distinct;
\item There are distinct indices $1\leq \ell_0,\ell_1,\dots,\ell_w\leq k$ such that
$x_1,\dots,x_{r-w}\in S^{down}_{\ell_0}$ and for every $1\leq a\leq w$, $x_{r-w+a}\in S^{down}_{\ell_a}$;
\item $|\cneigh_{G'}(x_1,\dots,x_r)|\leq s$.
\end{itemize}
Clearly, the probability of the event $\{Y_1,\dots,Y_u\in\cneigh_{G'}(x_1,\dots,x_r)\}$
can be bounded above by $\big(\frac{s}{|V^{up}_{G'}|}\big)^u$. Thus, the expected number of 
all ordered $r$--tuples satisfying the above three conditions and contained in the common
neighborhood of $Y_1,\dots,Y_u$, is at most
$$
\bigg(\frac{s}{|V^{up}_{G'}|}\bigg)^u\,\frac{k!}{(k-w-1)!}\frac{g^w\,g!}{(g-r+w)!}.
$$
\end{proof}

\begin{lemma}[Embedding into a block-structured graph]\label{ahgvhfgvfhgvcjhjh}
Let $G'= (V^{up}_{G'},V^{down}_{G'},E_{G'})$ be a block-structured bipartite graph
with parameters $(\delta,\gamma,k,g)$, let $n,w,u\geq 1$ with $n-w\geq 2$, 
and assume that
\begin{equation}\label{uyatcuewrcjqtcfqj}
k\cdot \frac{\delta}{2}(\gamma(1-\delta))^u\geq 2^w,\quad g\,(1-\delta)^{u+1}\geq 2^{n-w}.
\end{equation}
and
\begin{equation}\label{cagcckygiuycqtviy}
64\,\bigg(\frac{\delta}{4}\big(\gamma(1-\delta)\big)^u\bigg)^{-1}
\,\bigg(\frac{2^{n-1}}{|V^{up}_{G'}|}\bigg)^u
\big((1-\delta)^{u+1}\big)^{-n} \bigg(\frac{\delta}{2}\big(\gamma(1-\delta)\big)^u\bigg)^{-w-1}<2^{-n+1}.
\end{equation}
Then the hypercube $\{-1,1\}^n$ can be embedded into $G'$.
\end{lemma}
\begin{proof}
Let $\big(S^{up}_\ell\big)_{\ell=1}^k$,
$\big(S^{down}_\ell\big)_{\ell=1}^k$ be the corresponding
compatible sequences of subsets of vertices of $G'$.
Applying Lemma~\ref{adkjnfpiffiqfijfakdfk}, we get that
there are vertices $y_1,\dots,y_u\in V^{up}_{G'}$ such that
\begin{itemize}
\item $\big|\cneigh_{G'}(y_1,y_2,\dots,y_u)\cap S^{down}_{\ell}\big|\geq 
g\,(1-\delta)^{u+1}$
for at least $k\cdot \frac{\delta}{2}\big(\gamma(1-\delta)\big)^u$
indices $\ell$;
\item
the number of $n$--tuples $(x_1,\dots,x_n)\in \mathcal M_{G'}(n,w;y_1,\dots,y_u)$
with $|\cneigh_{G'}(x_1,\dots,x_n)|\leq 2^{n-1}$ is bounded above by
\begin{equation}\label{akjdnfapkfjnadfsfgrk}
\bigg(\frac{\delta}{4}\big(\gamma(1-\delta)\big)^u\bigg)^{-1}
\bigg(\frac{2^{n-1}}{|V^{up}_{G'}|}\bigg)^u\,\frac{k!}{(k-w-1)!}\frac{g^w\,g!}{(g-n+w)!}.
\end{equation}
\end{itemize}
Further, let $L\subset [k]$ be the subset of all indices $\ell$ such that
$\big|\cneigh_{G'}(y_1,y_2,\dots,y_u)\cap S^{down}_{\ell}\big|\geq 
g\,(1-\delta)^{u+1}$, and for every $\ell\in L$, let $\tilde S^{down}_{\ell}$ be any subset
of $\cneigh_{G'}(y_1,y_2,\dots,y_u)\cap S^{down}_{\ell}$ of cardinality
$\big\lceil g\,(1-\delta)^{u+1} \big\rceil$. Observe that in view of \eqref{uyatcuewrcjqtcfqj} we have
$$
|L|\geq 2^w\quad\mbox{and}\quad |\tilde S^{down}_{\ell}|\geq 2^{n-w}\;\;\mbox{for all $\ell\in L$}.
$$
Now, we construct a random mapping
$$f:\mathcal T=\big\{v\in \{-1,1\}^n:\;\mbox{vector $v$ has an odd number of $-1$'s}\big\}\to V^{down}_{G'}$$
as follows.
Let $(\mathcal T_b)_{b\in\{-1,1\}^w}$ be a partition of $\mathcal T$, where
$$
\mathcal T_b:=\big\{v\in \mathcal T:\,(v_{n-w+1},\dots,v_n)=b\big\},\quad b\in \{-1,1\}^w.
$$
Observe that $|\mathcal T_b|=2^{n-w-1}$.
Let $(Z_b)_{b\in\{-1,1\}^w}$ be the random $2^w$--tuple uniformly distributed on the collection
of all $2^w$--tuples of distinct indices from $L$. Then, conditioned on $(Z_b)_{b\in\{-1,1\}^w}$,
for every $b\in \{-1,1\}^w$ we let $f((v:\,v\in \mathcal T_b))$ be a random $2^{n-w-1}$--tuple
[conditionally] uniformly distributed on the collection of all $2^{n-w-1}$--tuples of distinct
elements of $\tilde S^{down}_{Z_b}$, and we require that
the random vectors $f((v:\,v\in \mathcal T_b))$, $b\in\{-1,1\}^w$
are [conditionally] mutually independent.
We recall at this point that the sets $\tilde S^{down}_{\ell}$, $\ell\in L$, are pairwise disjoint, and hence the mapping $f$
is injective everywhere on the probability space. Furthermore, the image of $f$ is contained within $\cneigh_{G'}(y_1,\dots,y_u)$.

Pick any vertex $v\in \{-1,1\}^n\setminus \mathcal T$ and let $v^{(1)},\dots,v^{(n)}$ be the neighbors
of $v$ in the hypercube ordered in such a way that
$$
\big(v^{(1)}_{n-w+1},\dots,v^{(1)}_{n}\big)=\dots=\big(v^{(n-w)}_{n-w+1},\dots,v^{(n-w)}_{n}\big)
=(v_{n-w+1},\dots,v_n),
$$
that is, for $b=(v_{n-w+1},\dots,v_n)\in\{-1,1\}^w$ we have $v^{(1)},\dots,v^{(n-w)}\in \mathcal T_b$.
Observe that the random $n$--tuple $\big(f(v^{(1)}),\dots,f(v^{(n)})\big)$
is uniformly distributed on the set of all $n$--tuples of distinct vertices $(x_1,\dots,x_n)$ such that 
there
are distinct indices $\ell_0,\ell_1,\dots,\ell_w\in L$ with
$x_1,\dots,x_{n-w}\in \tilde S^{down}_{\ell_0}$ and $x_{n-w+a}\in \tilde
S^{down}_{\ell_a}$ for every $1\leq a\leq w$.
Each $n$--tuple $(x_1,\dots,x_n)$ satisfying the above conditions belongs to the set
$\mathcal M_{G'}(n,w;y_1,\dots,y_u)$, and the total number of such $n$--tuples
is estimated from below by
$$
\frac{q!}{(q-w-1)!}\frac{p^w\,p!}{(p-n+w)!},
$$
where $q:=\big\lceil k\cdot \frac{\delta}{2}\big(\gamma(1-\delta)\big)^u\big\rceil$, and
$p:=\lceil g\,(1-\delta)^{u+1}\rceil$.
In view of the bound \eqref{akjdnfapkfjnadfsfgrk} and the
assumptions on the parameters, the probability that the number
of common neighbors of $f(v^{(1)}),\dots,f(v^{(n)})$ in $G'$ is less than $2^{n-1}$, is less than
\begin{align*}
&\bigg(\frac{\delta}{4}\big(\gamma(1-\delta)\big)^u\bigg)^{-1}
\frac{\big(\frac{2^{n-1}}{|V^{up}_{G'}|}\big)^u\,\frac{k!}{(k-w-1)!}\frac{g^w\,g!}{(g-n+w)!}}
{\frac{q!}{(q-w-1)!}\frac{p^w\,p!}{(p-n+w)!}}\\
&\leq \bigg(\frac{\delta}{4}\big(\gamma(1-\delta)\big)^u\bigg)^{-1}
\bigg(\frac{2^{n-1}}{|V^{up}_{G'}|}\bigg)^u
\big((1-\delta)^{u+1}\big)^{-w}
\bigg(\frac{k-w}{q-w}\bigg)^{w+1}\bigg(\frac{g-n+w+1}{p-n+w+1}\bigg)^{n-w}.
\end{align*}
In view of the first inequality in \eqref{uyatcuewrcjqtcfqj},
we have
\begin{align*}
\bigg(\frac{k-w}{q-w}\bigg)^{w+1}
&\leq \Big(\frac{\delta}{2}\big(\gamma(1-\delta)\big)^u\Big)^{-w-1}
\bigg(1+
  \frac{w}{k\cdot \frac{\delta}{2}\big(\gamma(1-\delta)\big)^u-w}\bigg)^{w+1}\\
  &\leq
  \Big(\frac{\delta}{2}\big(\gamma(1-\delta)\big)^u\Big)^{-w-1}
\bigg(1+
  \frac{w}{2^w-w}\bigg)^{w+1}\\
  & \leq 8\Big(\frac{\delta}{2}\big(\gamma(1-\delta)\big)^u\Big)^{-w-1},
\end{align*}
and, similarly, by the second inequality in \eqref{uyatcuewrcjqtcfqj},
\begin{align*}
\bigg(\frac{g-n+w+1}{p-n+w+1}\bigg)^{n-w}&\leq \big((1-\delta)^{-u-1}\big)^{n-w}
\bigg(\frac{g\,(1-\delta)^{u+1}}{g\,(1-\delta)^{u+1}-n+w+1}\bigg)^{n-w}\\
&\leq \big((1-\delta)^{-u-1}\big)^{n-w}
\bigg(\frac{2^{n-w}}{2^{n-w}-n+w+1}\bigg)^{n-w}\\
&\leq 8\big((1-\delta)^{-u-1}\big)^{n-w}.
\end{align*}
Hence,
\begin{align*}
\Prob&\big\{|\cneigh_{G'}(f(v^{(1)}),\dots,f(v^{(n)}))|\leq 2^{n-1}\big\}\\
&\leq 64
\bigg(\frac{\delta}{4}\big(\gamma(1-\delta)\big)^u\bigg)^{-1}
\bigg(\frac{2^{n-1}}{|V^{up}_{G'}|}\bigg)^u
\big((1-\delta)^{u+1}\big)^{-n}\Big(\frac{\delta}{2}\big(\gamma(1-\delta)\big)^u\Big)^{-w-1}<2^{-n+1},
\end{align*}
where the last inequality follows from the assumption \eqref{cagcckygiuycqtviy}.
Taking the union bound, we get that
with a positive probability, for every $v\in\{-1,1\}^n\setminus \mathcal T$, the set $f(\neigh_{\{-1,1\}^n}(v))$
has at least $2^{n-1}$ common neighbors in $G'$.
A simple iterative procedure then produces an embedding of $\{-1,1\}^n$ into $G'$,
completing the proof.
\end{proof}

\section{Trichotomy}\label{s:dich}

In this section, we show that any bipartite graph $G$ has at least one of the following three properties:
it either contains a large collection of non-condensed common neighborhoods of vertices,
or a large subgraph of higher edge density, or a large block-structured subgraph as defined in Section~\ref{s:blockstructured}.
As a corollary of the trichotomy we obtain the main result of the paper.

\begin{lemma}\label{akfhgvfiayfvewifyviyqvi}
Let $h,r>0$ be parameters.
Let $G'= (V^{up}_{G'},V^{down}_{G'},E_{G'})$ be a bipartite graph,
and let $Y_1,\dots,Y_r$ be i.i.d.\ uniform random vertices in $V^{up}_{G'}$.
Assume further that
$$
\Exp\,\big|\cneigh_{G'}(Y_1,\dots,Y_r)\big|\geq h.
$$
Then there is a subset $S\subset V^{down}_{G'}$ of size at least $h/2$
such that the induced subgraph of $G'$
on $V^{up}_{G'}\sqcup S$ has density at least $\big(\frac{h}{2 |V^{down}_{G'}|}\big)^{1/r}$.
\end{lemma}
\begin{proof}
For every vertex $v\in V^{down}_{G'}$, let $p_v\in[0,1]$ be the probability of the event $\{Y_1\in\neigh_{G'}(v)\}$, so that
$$
\sum_{v\in V^{down}_{G'}}p_v^r=
\Exp\,\big|\cneigh_{G'}(Y_1,\dots,Y_r)\big|\geq h
=|V^{down}_{G'}|\,\frac{h}{|V^{down}_{G'}|}.
$$
Let $S$ be the collection of all $v\in V^{down}_{G'}$ with $p_v^r\geq \frac{h}{2 |V^{down}_{G'}|}$,
and note that
$$
|S|\geq
\sum_{v\in S}p_v^r\geq \frac{h}{2}.
$$
The result follows.
\end{proof}

\begin{lemma}\label{aifgvajhcvaiy}
There is a universal constant $c_{\text{\tiny\ref{aifgvajhcvaiy}}}\in(0,1]$ with the following property.
Let $h,r,M>0$ and $p\in(0,1]$ be parameters.
Let $G'= (V^{up}_{G'},V^{down}_{G'},E_{G'})$ be a bipartite graph, and assume that
$(v_1,v_2)$ is a pair of vertices in $V^{down}_{G'}$ such that the collection of common neighborhoods
$\big\{\cneigh_{G'}(y_1,\dots,y_r):\,(y_1,\dots,y_r)\in \cneigh_{G'}(v_1,v_2)^r\big\}$
is $(p,M)$--condensed.
Further, assume that 
for i.i.d.\ uniform random vertices $Y_1,\dots,Y_r$ in $\cneigh_{G'}(v_1,v_2)$,
$$
\Exp\,\big|\cneigh_{G'}(Y_1,\dots,Y_r)\big|\leq h.
$$
Then there is a subset $S\subset V^{down}_{G'}$ of size at least $c_{\text{\tiny\ref{aifgvajhcvaiy}}}pM$
such that the induced subgraph of $G'$ on
$\cneigh_{G'}(v_1,v_2)\sqcup S$ has density at least
$$
\bigg(\frac{c_{\text{\tiny\ref{aifgvajhcvaiy}}} pM}{-h\log_2 (p/2)}\bigg)^{1/r}.
$$
\end{lemma}
\begin{proof}
Let $\tilde Y_1,\dots,\tilde Y_r$ be i.i.d.\ uniform random vertices in $\cneigh_{G'}(v_1,v_2)$
mutually independent with $Y_1,\dots,Y_r$. Further, for every $i\geq 0$
let $\Event_i$ be the event (measurable with respect to collection of variables $(Y_1,\dots,Y_r)$) that
$$
\Prob\big\{\big|\cneigh_{G'}(Y_1,\dots,Y_r)\cap \cneigh_{G'}(\tilde Y_1,\dots,\tilde Y_r)\big|
\geq M\;|\; Y_1,\dots,Y_r\big\}\in (2^{-i-1},2^{-i}].
$$
By the properties of the conditional probability and in view of the definition of a $(p,M)$--condensed collection,
we have
\begin{align*}
\Exp\,&\Prob\big\{\big|\cneigh_{G'}(Y_1,\dots,Y_r)\cap \cneigh_{G'}(\tilde Y_1,\dots,\tilde Y_r)\big|
\geq M\;|\; Y_1,\dots,Y_r\big\}\\
&=
\Prob\big\{\big|\cneigh_{G'}(Y_1,\dots,Y_r)\cap \cneigh_{G'}(\tilde Y_1,\dots,\tilde Y_r)\big|
\geq M\big\}\geq p.
\end{align*}
On the other hand, for every non-negative random variable $\xi$ we have
$$
\Exp\,\xi\leq \sum_{i\geq 0}2^{-i}\,\Prob\big\{\xi \in (2^{-i-1},2^{-i}]\big\}.
$$
It follows that
$$
\sum_{i\geq 0}2^{-i}\,\Prob(\Event_i)\geq p,
$$
and hence there is an index $0\leq i_0\leq -2\log_2 (p/4)$ with
$$
2^{-{i_0}}\,\Prob(\Event_{i_0})\geq \frac{p}{-2\log_2 (p/4)}.
$$
Using the assumption that $\Exp\,\big|\cneigh_{G'}(Y_1,\dots,Y_r)\big|\leq h$, we get a bound on the conditional expectation
$$
\Exp\,\big[\big|\cneigh_{G'}(Y_1,\dots,Y_r)\big|\;|\;\Event_{i_0}\big]\leq h\cdot \frac{1}{\Prob(\Event_{i_0})}
\leq h\cdot \frac{-2\log_2 (p/4)}{2^{i_0} p}.
$$
Thus, there is a collection of [non-random] vertices $y_1,\dots,y_r\in \cneigh_{G'}(v_1,v_2)$
with
$$
\Prob\big\{\big|\cneigh_{G'}(y_1,\dots,y_r)\cap \cneigh_{G'}(\tilde Y_1,\dots,\tilde Y_r)\big|
\geq M\big\}\in (2^{-{i_0}-1},2^{-{i_0}}]
$$
and such that $\big|\cneigh_{G'}(y_1,\dots,y_r)\big|\leq h\cdot \frac{-2\log_2 (p/4)}{2^{i_0} p}$.

For every vertex $v\in \cneigh_{G'}(y_1,\dots,y_r)$, let $p_v$ be the probability of the event $\{\tilde Y_1\in
\neigh_{G'}(v)\}$.
By our assumptions,
$$
\sum_{v\in \cneigh_{G'}(y_1,\dots,y_r)}p_v^r=
\Exp\,\big|\cneigh_{G'}(y_1,\dots,y_r)\cap \cneigh_{G'}(\tilde Y_1,\dots,\tilde Y_r)\big|
\geq 2^{-{i_0}-1}M.
$$
Let $S$ be the set of all vertices $v\in \cneigh_{G'}(y_1,\dots,y_r)$ such that
$$p_v^r\geq \frac{2^{-i_0-2}M}{h\cdot \frac{-2\log_2 (p/4)}{2^{i_0} p}}
=\frac{pM}{-8h\log_2 (p/4)},$$
and note that
$$
|S|\geq \sum_{v\in S}p_v^r\geq 2^{-{i_0}-2}M.
$$
The induced subgraph of $G'$ on $\cneigh_{G'}(v_1,v_2)\sqcup S$ has density at least
$$
\bigg(\frac{pM}{-8h\log_2 (p/4)}\bigg)^{1/r},
$$
completing the proof.
\end{proof}

\begin{prop}[The main structural proposition]\label{vviyagcvuytcvuqtycv}
Let $G= (V^{up}_{G},V^{down}_{G},E_{G})$ be a non-empty
bipartite graph of density at least $\alpha\in(0,1]$. 
Further, let $\alpha_0\leq\alpha/2$, $\mu\in(0,\alpha_0/2]$, $M\geq 1$, $r\geq 2$, and $p\in(0,1]$ be parameters and
assume that $((1-\mu)\alpha)^2\,|V^{up}_{G}|\geq r^2$.
Then at least one of the following is true:
\begin{itemize}
\item[(a)] There is a subgraph $G'= (V^{up}_{G'},V^{down}_{G'},E_{G'})$ of $G$ of density at least $(1-\mu)\alpha$,
with $V^{up}_{G'}:=V^{up}_{G}$ and $V^{down}_{G'}:=V^{down}_{G}$,
and an $(\alpha_0,(1-\mu)\alpha,\mu,r,C_{\text{\tiny\ref{dfalnkfjnefwehfbqlwuh}}}(\alpha_0,\mu)\, r^3)$--standard
vertex pair $(v_1,v_2)$ in $V^{down}_{G'}$
such that the collection of neighborhoods
$$\big\{\cneigh_{G'}(y_1,\dots,y_r):\,(y_1,\dots,y_r)\in \cneigh_{G'}(v_1,v_2)^r\big\}$$
is {\bf not} $(p,M)$--condensed.
\item[(b)] There is a subgraph $G'= (V^{up}_{G'},V^{down}_{G'},E_{G'})$ of $G$ with
$|V^{up}_{G'}|\geq \frac{\alpha^2}{2}|V^{up}_{G}|$, $|V^{down}_{G'}|\geq h/2$ of density at least $\big(\frac{h}{2 |V^{down}_{G}|}\big)^{1/r}$.
\item[(c)] There is a block-structured subgraph $G'= (V^{up}_{G'},V^{down}_{G'},E_{G'})$ of $G$
with $V^{up}_{G'}:=V^{up}_{G}$ and
with parameters
$$
\bigg(1-\bigg(\frac{c_{\text{\tiny\ref{aifgvajhcvaiy}}} pM}{-h\log_2 (p/2)}\bigg)^{1/r},
(1-\mu)^3\alpha^2,
\bigg\lceil\frac{\mu\alpha\,|V^{down}_{G}|}{\lceil
c_{\text{\tiny\ref{aifgvajhcvaiy}}}pM\rceil}\bigg\rceil,
\lceil c_{\text{\tiny\ref{aifgvajhcvaiy}}}pM\rceil\bigg).
$$
\end{itemize}
\end{prop}
\begin{proof}
The proof is accomplished via an iterative process.

Set
$\ell:=1$, $G^{(1)}:=G$,
so that $G^{(1)}$ has density at least $\alpha$.

\bigskip

\noindent {\bf{}Beginning of cycle}

\medskip

At the start of the cycle, we assume that we are given a bipartite subgraph $G^{(\ell)}$ of $G$
having the same vertex set as $G$, with the edge density
$$
\alpha^{(\ell)}\geq \alpha-\frac{(\ell-1)\cdot \lceil
c_{\text{\tiny\ref{aifgvajhcvaiy}}}pM\rceil}{|V^{down}_{G}|}\geq (1-\mu)\alpha.
$$
Applying Lemma~\ref{dfalnkfjnefwehfbqlwuh},
we obtain an $(\alpha_0,(1-\mu)\alpha,\mu,r,C_{\text{\tiny\ref{dfalnkfjnefwehfbqlwuh}}} r^3)$--standard
ordered pair $(v_1^{(\ell)},v_2^{(\ell)})$ in $V^{down}_{G^{(\ell)}}$.
If the collection of neighborhoods
$\big\{\cneigh_{G^{(\ell)}}(y_1,\dots,y_r):\,(y_1,\dots,y_r)\in \cneigh_{G^{(\ell)}}(v_1^{(\ell)},v_2^{(\ell)})^r\big\}$
is not $(p,M)$--condensed then we arrive at the option (a) and complete the proof.

Otherwise, if the i.i.d.\ uniform random elements $Y_1^{(\ell)},\dots,Y_r^{(\ell)}$ of
$\cneigh_{G^{(\ell)}}(v_1^{(\ell)},v_2^{(\ell)})$ satisfy
$$
\Exp\,\big|\cneigh_{G^{(\ell)}}(Y_1^{(\ell)},\dots,Y_r^{(\ell)})\big|\geq h
$$
then, by Lemma~\ref{akfhgvfiayfvewifyviyqvi},
there is an induced subgraph $G'= (V^{up}_{G'},V^{down}_{G'},E_{G'})$ of $G^{(\ell)}$ with
$V^{up}_{G'}:=\cneigh_{G^{(\ell)}}(v_1^{(\ell)},v_2^{(\ell)})$ and $|V^{down}_{G'}|\geq h/2$
of density at least $\big(\frac{h}{2 |V^{down}_{G}|}\big)^{1/r}$, implying (b) and completing the proof.

Otherwise, $\big\{\cneigh_{G^{(\ell)}}(y_1,\dots,y_r):\,(y_1,\dots,y_r)\in \cneigh_{G^{(\ell)}}(v_1^{(\ell)},v_2^{(\ell)})^r\big\}$
is $(p,M)$--condensed and
$$
\Exp\,\big|\cneigh_{G^{(\ell)}}(Y_1,\dots,Y_r)\big|\leq h.
$$
Lemma~\ref{aifgvajhcvaiy} then implies that
there is a subset $S_\ell^{down}\subset V^{down}_{G^{(\ell)}}$ of non-isolated vertices of size $\lceil
c_{\text{\tiny\ref{aifgvajhcvaiy}}}pM\rceil$
such that the induced subgraph of $G^{(\ell)}$ on
$\cneigh_{G^{(\ell)}}(v_1^{(\ell)},v_2^{(\ell)})\sqcup S_\ell^{down}$ has density at least
$$
\bigg(\frac{c_{\text{\tiny\ref{aifgvajhcvaiy}}} pM}{-h\log_2 (p/2)}\bigg)^{1/r}.
$$
Set $S_\ell^{up}:=\cneigh_{G^{(\ell)}}(v_1^{(\ell)},v_2^{(\ell)})$ and note that
$|S_\ell^{up}|\geq (1-\mu)((1-\mu)\alpha)^2|V^{up}_{G}|$.
If $\ell\geq \frac{\mu\alpha\,|V^{down}_{G}|}{\lceil
c_{\text{\tiny\ref{aifgvajhcvaiy}}}pM\rceil}$ then we exit the cycle and proceed with the rest of the proof.
Otherwise, we 
let $G^{(\ell+1)}= (V^{up}_{G^{(\ell+1)}},V^{down}_{G^{(\ell+1)}},E_{G^{(\ell+1)}})$
be the subgraph of $G^{(\ell)}$
obtained from $G^{(\ell)}$ by removing all edges adjacent to vertices in the set $S_\ell^{down}$.
By our construction process, the edge density $\alpha^{(\ell+1)}$ of $G^{(\ell+1)}$ satisfies
$$
\alpha^{(\ell+1)}\geq
\alpha-
\frac{\ell\cdot \lceil
c_{\text{\tiny\ref{aifgvajhcvaiy}}}pM\rceil}{|V^{down}_{G}|}\geq (1-\mu)\alpha.
$$
Set
$$
\ell:=\ell+1
$$
and return to the beginning of the cycle.

\noindent {\bf{}End of cycle}

\bigskip

Set
$$
k:=\bigg\lceil\frac{\mu\alpha\,|V^{down}_{G}|}{\lceil
c_{\text{\tiny\ref{aifgvajhcvaiy}}}pM\rceil}\bigg\rceil.
$$
Upon completion of the cycle, we have two collections of subsets $(S_\ell^{up})_{\ell=1}^k$
and $(S_\ell^{down})_{\ell=1}^k$ such that for each $\ell$, $|S_\ell^{down}|=\lceil
c_{\text{\tiny\ref{aifgvajhcvaiy}}}pM\rceil$, $|S_\ell^{up}|\geq (1-\mu)^3\alpha^2|V^{up}_{G}|$, and
the induced subgraph of $G$ on
$S_\ell^{up}\sqcup S_\ell^{down}$ has density at least
$$
\bigg(\frac{c_{\text{\tiny\ref{aifgvajhcvaiy}}} pM}{-h\log_2 (p/2)}\bigg)^{1/r}.
$$
Observe further that the sets $(S_\ell^{down})_{\ell=1}^k$ are necessarily disjoint by construction.
Thus, $G$ contains a block-structured subgraph with parameters
$$
\bigg(1-\bigg(\frac{c_{\text{\tiny\ref{aifgvajhcvaiy}}} pM}{-h\log_2 (p/2)}\bigg)^{1/r},
(1-\mu)^3\alpha^2,
\bigg\lceil\frac{\mu\alpha\,|V^{down}_{G}|}{\lceil
c_{\text{\tiny\ref{aifgvajhcvaiy}}}pM\rceil}\bigg\rceil,
\lceil c_{\text{\tiny\ref{aifgvajhcvaiy}}}pM\rceil\bigg).
$$
The result follows.
\end{proof}

\begin{proof}[Proof of Theorem~\ref{thmain}]
Let
$$
\mu:=10^{-10},\quad\alpha:=\frac{1}{2},\quad \alpha_0:=0.1,\quad c:=c'-100\mu,
$$
where $c'= 0.03657...$ is the positive solution of the quadratic equation $64(c')^2+25c'-1=0$.
We will assume that $n$ is sufficiently large and let $m:=2^{n-1}$. 
Let $G=(V^{up}_{G},V^{down}_{G},E_{G})$ be a bipartite
graph of density at least $\alpha=\frac{1}{2}$, where $|V^{up}_{G}|=|V^{down}_{G}|=\lceil 2^{2n-cn}\rceil$.
Define
$$
M:=\bigg\lfloor \frac{c_{\text{\tiny\ref{kjfaoifboqfuifhbluhb}}}^2\,((1-\mu)\alpha)^{4n}|V^{down}_{G}|^4}{2^9 C_{\text{\tiny\ref{dfalnkfjnefwehfbqlwuh}}}^2\,m^3 \,n^6 
\cdot 400\log^2 |V^{down}_{G}|}\bigg\rfloor=(1-\mu)^{4n}\,2^{n-4cn+o(n)}
$$
and
$$
p:=\bigg(\frac{c_{\text{\tiny\ref{kjfaoifboqfuifhbluhb}}}\,((1-\mu)\alpha)^{2n}\,|V^{down}_{G}|^2}
{16\cdot 3^7 C_{\text{\tiny\ref{dfalnkfjnefwehfbqlwuh}}}\, n^3 m^2
\cdot 20\log |V^{down}_{G}|}\bigg)^2=(1-\mu)^{4n}\,2^{-4cn+o(n)},
$$
and note that with this definition and our assumption that $n$ is large,
we have $M\geq 1$ and $n^2\leq p\cdot ((1-\mu)\alpha)^n |V^{down}_{G}|$.
Further, set
$$
u:=\lfloor \sqrt{n}\rfloor,
$$
define $w\in \N$ via the relation
$$
n-w=\Big\lfloor\log_2\Big(\lceil c_{\text{\tiny\ref{aifgvajhcvaiy}}}pM\rceil\cdot 2^{-u-1}\Big)\Big\rfloor
=n-8cn+8n\log_2(1-\mu)+o(n),
$$
and define $h>0$ via the formula
$$
\bigg(\frac{c_{\text{\tiny\ref{aifgvajhcvaiy}}} pM}{-h\log_2 (p/2)}\bigg)^{(n+w)/n}
=2^{2w+cn-n}(1-\mu)^{-3n},
$$
so that
$$
h=(1-\mu)^{8n}\,2^{n-8cn+o(n)}\big(2^{17cn-16n\log_2(1-\mu)-n}(1-\mu)^{-3n}\big)^{-\frac{1}{1+8c-8\log_2(1-\mu)}}.
$$
To evaluate the magnitude of $h$, note that the right hand side of the last identity can be crudely bounded from below by
$$
\exp(-100\mu n+o(n))\,2^{n-8cn}\big(2^{17cn-n}\big)^{-\frac{1}{1+8c}}
=\exp(-100\mu n+o(n))\,2^{n+n\cdot\frac{1-25c-64c^2}{1+8c}},
$$
where, in view of the definition of $c$, $\frac{1-25c-64c^2}{1+8c}\geq \frac{25\cdot 100\mu}{1+8c}\geq 1250\mu$,
and hence $h\geq 2^{n+1000\mu n+o(n)}$.

Applying Proposition~\ref{vviyagcvuytcvuqtycv} with $r:=n$,
we get that at least one of the conditions (a), (b), (c) there holds.
Below, we deal with each of the three possibilities.
\begin{itemize}

\item[(a)] 
In this case, we are given a subgraph
$G'= (V^{up}_{G'},V^{down}_{G'},E_{G'})$ of $G$ of density at least $(1-\mu)\alpha$,
with $V^{up}_{G'}:=V^{up}_{G}$ and $V^{down}_{G'}:=V^{down}_{G}$,
and an $(\alpha_0,(1-\mu)\alpha,\mu,r,C_{\text{\tiny\ref{dfalnkfjnefwehfbqlwuh}}}(\alpha_0,\mu)\, r^3)$--standard
vertex pair $(v_1,v_2)$ in $V^{down}_{G'}$
such that the collection of neighborhoods
$$\big\{\cneigh_{G'}(y_1,\dots,y_r):\,(y_1,\dots,y_r)\in \cneigh_{G'}(v_1,v_2)^r\big\}$$
is not $(p,M)$--condensed.
Observe that the assumptions of Lemma~\ref{fljknfefjwfpiwjff} (with $\alpha$ replaced with $(1-\mu)\alpha$)
hold.
Applying Lemma~\ref{fljknfefjwfpiwjff}, we get that $Q_n$ can be embedded into $G$.

\item[(b)] In this case,
there is a subgraph $G'= (V^{up}_{G'},V^{down}_{G'},E_{G'})$ of $G$ with
$|V^{up}_{G'}|\geq \frac{1}{8}|V^{up}_{G}|$, $|V^{down}_{G'}|\geq h/2$,
of density at least $\big(\frac{h}{2 |V^{down}_{G}|}\big)^{1/n}$.
Since $h\geq 2^{n+1000\mu n+o(n)}$ and assuming $n$ is large, we have
$$
|V^{up}_{G'}|\geq 2^{n+\tilde c n}\cdot \frac{2 |V^{down}_{G}|}{h},\quad
|V^{down}_{G'}|\geq 2^{n+\tilde c n},\quad \bigg(\frac{h}{2 |V^{down}_{G}|}\bigg)^{1/n}\geq \tilde c,
$$
for a universal constant $\tilde c>0$.
Applying Corollary~\ref{aljfhbefwhbfoufbowfub}, we obtain an embedding of $Q_n$ into $G$.

\item[(c)] In this case,
$G$ contains a block-structured subgraph $G'=(V^{up}_{G'},V^{down}_{G'},E_{G'})$
with $V^{up}_{G'}=V^{up}_{G}$ and with parameters
\begin{align*}
(\delta,\gamma,k,g):=\bigg(&1-\bigg(\frac{c_{\text{\tiny\ref{aifgvajhcvaiy}}} pM}{-h\log_2 (p/2)}\bigg)^{1/n},
(1-\mu)^3\alpha^2,
\bigg\lceil\frac{\mu\alpha\,|V^{down}_{G}|}{\lceil
c_{\text{\tiny\ref{aifgvajhcvaiy}}}pM\rceil}\bigg\rceil,
\lceil c_{\text{\tiny\ref{aifgvajhcvaiy}}}pM\rceil\bigg)\\
=\bigg(&1-\bigg(\frac{c_{\text{\tiny\ref{aifgvajhcvaiy}}} pM}{-h\log_2 (p/2)}\bigg)^{1/n},
\frac{(1-\mu)^3}{4},\\
&(1-\mu)^{-8n}\,2^{n+7cn+o(n)},
(1-\mu)^{8n}\,2^{n-8cn+o(n)}\bigg).
\end{align*}
In view of our definition of $u,w,h$ and assuming that $n$ is sufficiently large, we have
$$
k\cdot \frac{\delta}{2}(\gamma(1-\delta))^u\geq 2^w,\quad g\,(1-\delta)^{u+1}\geq 2^{n-w}.
$$
and
$$
64\,\bigg(\frac{\delta}{4}\big(\gamma(1-\delta)\big)^u\bigg)^{-1}
\,\bigg(\frac{2^{n-1}}{|V^{up}_{G}|}\bigg)^u
\big((1-\delta)^{u+1}\big)^{-n} \bigg(\frac{\delta}{2}\big(\gamma(1-\delta)\big)^u\bigg)^{-w-1}<2^{-n+1}.
$$
Hence, applying Lemma~\ref{ahgvhfgvfhgvcjhjh}, we obtain an embedding of $Q_n$ into $G$.
The proof is complete.
\end{itemize}

\end{proof}

\bigskip

\begin{Remark}[Algorithmic perspective]
It should be expected that our construction of the hypercube embedding can be turned into a randomized algorithm
which runs in time polynomial in the number of vertices of the ambient graph.
We leave this as an open question.
\end{Remark}

\appendix

\section{The random graph $\Gamma_{\varepsilon,n}$}\label{lfuhb93476ajshbflj}

The goal of this section is to show that for every $\varepsilon>0$ and any sufficiently large $n$,
there exists a bipartite graph on $\Theta(2^{2n-\varepsilon n})$ vertices and with the edge
density $1/2$
which does not admit embedding of the hypercube $Q_n$ via the randomized procedure {\bf(I)--(II)} from the introduction.

\bigskip

Fix a small parameter $\varepsilon>0$, an integer $n\in\N$.
Let $V^{up}_{\varepsilon,n}$ and $V^{down}_{\varepsilon,n}$ be two disjoint sets with
$|V^{up}_{\varepsilon,n}|=|V^{down}_{\varepsilon,n}|= 2\cdot\big\lceil 2^{n-\varepsilon n/2}\big\rceil^2$,
and assume that
the set $V^{down}_{\varepsilon,n}$
is partitioned into $2\cdot\lceil 2^{n-\varepsilon n/2}\rceil$ subsets $V^{down}_{\varepsilon,n}(i)$, $1\leq i\leq 2\cdot\lceil 2^{n-\varepsilon n/2}\rceil$,
of size $\lceil 2^{n-\varepsilon n/2}\rceil$ each.
Further, for each $v\in V^{up}_{\varepsilon,n}$, let $I_v$ be 
a uniform random $\lceil 2^{n-\varepsilon n/2}\rceil$--subset of $\big\{1,\dots,2\cdot\lceil 2^{n-\varepsilon n/2}\rceil\big\}$,
and assume that the sets $I_v$, $v\in V^{up}_{\varepsilon,n}$, are mutually independent.
Consider a random bipartite graph $\Gamma_{\varepsilon,n}=(V^{up}_{\varepsilon,n},V^{down}_{\varepsilon,n},E_{\varepsilon,n})$,
where the edge set $E_{\varepsilon,n}$ is comprised of all unordered pairs of vertices $\{v,w\}$, $v\in V^{up}_{\varepsilon,n}$,
$w\in \bigcup_{i\in I_v}V^{down}_{\varepsilon,n}(i)$.
Note that the edge density of $\Gamma_{\varepsilon,n}$ is $1/2$ everywhere on the probability space.
Our goal is to prove

\begin{prop}
For every $\varepsilon\in(0,1]$ there is $n_\varepsilon\in\N$ such that
given $n\geq n_\varepsilon$ and the random bipartite graph
$\Gamma_{\varepsilon,n}=(V^{up}_{\varepsilon,n},V^{down}_{\varepsilon,n},E_{\varepsilon,n})$,
with a positive probability $\Gamma_{\varepsilon,n}$
has the following property. 
For every subset $S$ of $V^{up}_{\varepsilon,n}$ with $|S|\geq 2^{n-1}$
\begin{equation}\label{alfkjhbfwfhjbljqhbl}
\begin{split}
&\mbox{there is a subset $T=T(S)\subset V^{down}_{\varepsilon,n}$
with $|T|\leq 2^{n-\varepsilon n/4}$}\\
&\mbox{such that for at least half of $n$--tuples
of vertices in $S$,}\\
&\mbox{the common neighborhood of the $n$--tuple is contained in $T$.}
\end{split}
\end{equation}
\end{prop}
\begin{proof}
Fix any $\varepsilon\in(0,1)$. We will assume that $n$ is large.
Fix for a moment any subset $S$ of $V^{up}_{\varepsilon,n}$ with $|S|\geq 2^{n-1}$.
We will estimate the probability of the event $\Event_S$ that $S$ {\it does not} satisfy \eqref{alfkjhbfwfhjbljqhbl}. 
For each $i\in \big\{1,\dots,2\cdot\lceil 2^{n-\varepsilon n/2}\rceil\big\}$,
let $\delta_i\in[0,1]$ be the [random] number defined by
$$
\delta_i=\frac{|\{v\in S:\;v\mbox{ is adjacent to $V^{down}_{\varepsilon,n}(i)$}\}|}{|S|}.
$$
Note that $\delta_i^n$ can be viewed as the proportion of ordered $n$--tuples of vertices in $S$
(with repetitions allowed) comprising $V^{down}_{\varepsilon,n}(i)$ in their common neighborhood. 
Define the random set $U$ as
$$
U:=\bigg\{i\in\{1,\dots,2\cdot\lceil 2^{n-\varepsilon n/2}\rceil\}:\;
\delta_i^n\geq \frac{1}{4\cdot\lceil 2^{n-\varepsilon n/2}\rceil}
\bigg\}.
$$
Condition for a moment on any realization of $\Gamma_{\varepsilon,n}$ such that
$|U|\leq \frac{2^{n-\varepsilon n/4}}{\lceil 2^{n-\varepsilon n/2}\rceil}$.
Since
$$
\sum_{i\in [\lceil 2^{n-\varepsilon n/2}\rceil]\setminus U}
\delta_i^n\leq \frac{1}{2},
$$
for at least half of the ordered $n$--tuples $(v_1,\dots,v_n)$ of vertices in $S$ the set difference
$$
\cneigh_{\Gamma_{\varepsilon,n}}(v_1,\dots,v_n)\setminus \bigcup\limits_{i\in U}V^{down}_{\varepsilon,n}(i)
$$
is empty, implying that $S$ satisfies the property \eqref{alfkjhbfwfhjbljqhbl} with
$T:=\bigcup\limits_{i\in U}V^{down}_{\varepsilon,n}(i)$. Thus, necessarily
$$
\Event_S\subset\bigg\{|U|\geq \frac{2^{n-\varepsilon n/4}}{\lceil 2^{n-\varepsilon n/2}\rceil}\bigg\}.
$$
Fix any non-random set $L\subset [\lceil 2^{n-\varepsilon n/2}\rceil]$ of size at least
$\frac{2^{n-\varepsilon n/4}}{\lceil 2^{n-\varepsilon n/2}\rceil}$.
The probability of the event $\{U=L\}$ can be estimated from above as
$$
\Prob\big\{U=L\big\}
\leq\Prob\Bigg\{\frac{|\{v\in S:\;i\in I_v\}|}{|S|}\geq \frac{1}{\big(4\cdot\lceil 2^{n-\varepsilon n/2}\rceil\big)^{1/n}}
\mbox{ for every }i\in L\Bigg\}.
$$
Let $(b_{v,i})$, $v\in S$, $i\in L$, be a collection of i.i.d.\ Bernoulli($1/2$) random variables. Then, in view of the definition
of the random sets $I_v$, we can write
\begin{align*}
&\Prob\Bigg\{\frac{|\{v\in S:\;i\in I_v\}|}{|S|}\geq \frac{1}{\big(4\cdot\lceil 2^{n-\varepsilon n/2}\rceil\big)^{1/n}}
\mbox{ for every }i\in L\Bigg\}\\
&\hspace{1cm}\leq 
\Prob\Bigg\{\sum_{v\in S}b_{v,i}\geq \frac{|S|}{\big(4\cdot\lceil 2^{n-\varepsilon n/2}\rceil\big)^{1/n}}
\mbox{ for every }i\in L\Bigg\}\\
&\hspace{2cm}\cdot \Prob\bigg\{\sum_{i=1}^{2\cdot\lceil 2^{n-\varepsilon n/2}\rceil}b_{v,i}=\lceil 2^{n-\varepsilon n/2}\rceil
\mbox{ for every }v\in S\bigg\}^{-1}.
\end{align*}
A crude estimate
$$
\Prob\bigg\{\sum_{i=1}^{2\cdot\lceil 2^{n-\varepsilon n/2}\rceil}b_{v,i}=\lceil 2^{n-\varepsilon n/2}\rceil
\mbox{ for every }v\in S\bigg\}^{-1}
\leq 2^{n|S|}
$$
will be sufficient for our purposes. Further, the Hoeffding inequality implies
$$
\Prob\Bigg\{\sum_{v\in S}b_{v,i}\geq \frac{|S|}{\big(4\cdot\lceil 2^{n-\varepsilon n/2}\rceil\big)^{1/n}}
\mbox{ for every }i\in L\Bigg\}
\leq \exp\big(-c|S|\cdot |L|\big)
$$
for some $c=c(\varepsilon)>0$.
Hence,
$$
\Prob\big\{U=L\big\}\leq 2^{n|S|}\exp\big(-c|S|\cdot |L|\big)\leq \exp\big(-c|S|\cdot |L|/2\big),
$$
where in the last inequality we assumed that $n$ is sufficiently large.
A union bound estimate gives
\begin{align*}
\Prob\bigg\{|U|\geq \frac{2^{n-\varepsilon n/4}}{\lceil 2^{n-\varepsilon n/2}\rceil}\bigg\}
&=\sum_{m\geq 2^{n-\varepsilon n/4}/\lceil 2^{n-\varepsilon n/2}\rceil}
\Prob\big\{|U|=m\big\}\\
&\leq \sum_{m\geq 2^{n-\varepsilon n/4}/\lceil 2^{n-\varepsilon n/2}\rceil}
\exp\big(-c|S|\cdot m/2\big)\bigg(\frac{2e\,\lceil 2^{n-\varepsilon n/2}\rceil}{m}\bigg)^m\\
&\leq \exp\big(-c|S|\cdot 2^{\varepsilon n/4}/4\big),
\end{align*}
again under the assumption of large $n$.
We conclude that the probability that the set $S$ does not satisfy \eqref{alfkjhbfwfhjbljqhbl},
is bounded above by $\exp\big(-c'|S|\cdot 2^{\varepsilon n/4}\big)$,
for some $c'=c'(\varepsilon)>0$. Another union bound estimate implies
\begin{align*}
\Prob&\big\{\mbox{$S$ does not satisfy \eqref{alfkjhbfwfhjbljqhbl} for some $S\subset
V^{up}_{\varepsilon,n}$ of size at least $2^{n-1}$}\big\}\\
&\leq \sum\limits_{m\geq 2^{n-1}}\exp\big(-c'm\cdot 2^{\varepsilon n/4}\big)\bigg(\frac{e\cdot 2^{2n}}{m}\bigg)^m<1.
\end{align*}
Hence, there is a realization of $\Gamma_{\varepsilon,n}$ with the required properties.
\end{proof}

\end{document}